\documentclass[a4paper,11pt]{article}
\usepackage{graphicx,bbm,amssymb,amsmath,amsfonts,fullpage,amsthm}

\usepackage{stmaryrd}
\usepackage[pdftex]{hyperref}

\newcommand{\p}{\mathbb{P}}

\newcommand{\R}{\mathbb{R}}

\newcommand{\N}{\mathbb{N}}
\newcommand{\E}{\mathbb{E}}

\newcommand{\M}{\mathcal{M}}
\newcommand{\F}{\mathcal{F}}

\newcommand{\Var}{\textrm{Var}}

\newtheorem{theo}{Theorem}[section]
\newtheorem{coro}[theo]{Corollary}
\newtheorem{prop}[theo]{Proposition}
\newtheorem{definition}[theo]{Definition}
\newtheorem{lemma}[theo]{Lemma}

\title{Adaptive non-asymptotic confidence balls in density estimation.}
\author{Matthieu Lerasle\footnote{Instituto de Matem\'atica e Estat\'istica -USP 
Granted by Fapesp Processo 2009/09494-0,}}

\begin{document}

\maketitle

\hspace{1cm}\begin{minipage}{12cm}
\begin{center}
Abstract:
\end{center}
{\small We build confidence balls for the common density $s$ of a real valued sample $X_1,...,X_n$. We use resampling methods to estimate the projection of $s$ onto finite dimensional linear spaces and a model selection procedure to choose an optimal approximation space. The covering property is ensured for all $n\geq2$ and the balls are adaptive over a collection of linear spaces.}
\end{minipage}
\vspace{0.5cm}



\section{Introduction}
In this paper, we discuss the problem of adaptive confidence balls, from a non-asymptotic point of view, in the particular context of density estimation. Let $S$ be a set of densities with respect to the Lebesgue measure $\mu$ on $\R$. Given an i.i.d sample $X_{1:n}=(X_1,...,X_n)$ and a confidence level $\beta\in(0,1)$, a confidence set (hereafter CS) $\hat{B}_{\beta}(X_{1:n})$ on $S$ is a subset of $S$ satisfying the following covering property:
\begin{equation}\label{eq:couv}
\forall s\in S,\; \p_s\left(s\in \hat{B}_{\beta}(X_{1:n})\right)\geq 1-\beta
\end{equation}
where, for all $s$ in $S$, $\p_s$ denotes the distribution of $X_{1:n}$ when the marginals have common density $s$. All the CS considered in this paper are $L^2$-balls, centered on estimators $\hat{s}$ of $s$, and with random radius $\hat{\rho}_{\beta}$. The quality of a CS is measured with the quantiles of $\hat{\rho}_{\beta}$. We are looking for adaptive CS, which means that, given a collection $(S_m)_{m\in\M_n}$ of subsets of $S$, $\hat{\rho}_{\beta}$ should be as small as possible over all the sets $(S_m)_{m\in\M_n}$.\\
This problem was mostly considered in regression frameworks, see among others  Li \cite{Li89}, Lepski \cite{Le99}, Juditski $\&$ Lepski \cite{JL01}, Hoffmann $\&$ Lepski \cite{HL02}, Juditski $\&$ Lambert-Lacroix \cite{JLL03}, Baraud \cite{Ba04}, Beran \cite{Be00}, Beran $\&$ D\"umbgen \cite{BD98}, Cai $\&$ Low \cite{CL06}, Genovese $\&$ Wassermann \cite{GW08, GW05}. Robins $\&$ van der Vaart \cite{RV06} considered a more general Hilbertian framework that includes in particular density estimation and some regression frameworks.\\
Our adaptive balls are derived from a model selection procedure, which is essentially the one of Baraud \cite{Ba04}. We start with a collection of linear spaces $(S_m)_{m\in\M_n}$ and associate to each of these, the projection estimator $\hat{s}_m$ of $s$ and some positive number $\hat{\rho}(m)$. The $\hat{\rho}(m)$'s are suitably calibrated to satisfy the property that, with probability close to one the distance between $s$ and its projection estimator $\hat{s}_m$ is not larger than $\hat{\rho}(m)$. We then select $\hat{m}$ as the minimizer of $\hat{\rho}(m)$ and define the confidence ball as the $L^2$-ball centered at $\hat{s}_{\hat{m}}$ of radius  $\hat{\rho}(\hat{m})$.\\
We use two different ingredients to compute $\hat{\rho}(m)$. The first one is a resampling estimator of $\|s_m-\hat{s}_m\|^2$, where $s_m$ denotes the projection of $s$ onto $S_m$. It is naturally derived from Efron's heuristic (see Efron \cite{Ef79}), in the same way as Arlot, Blanchard $\&$ Roquain \cite{ABR07}. This allows us in particular to keep all the sample to build $\hat{s}_{m}$. This is an improvement compared with Robins $\&$ van der Vaart \cite{RV06} or Cai $\&$ Low \cite{CL06}, who cut the sample into two parts, the first one being used to build an estimator $\hat{s}$ of $s$ and the other to evaluate the distance $\|\hat{s}-s\|^2$.\\
The second ingredient is an estimator of $\|s-s_m\|^2$, based on U-statistics, as in Laurent \cite{La96,La05}. The proofs are handled thanks to a concentration inequality for $U$-statistics, derived from Houdr\'e $\&$ Reynaud-Bouret \cite{HRB03}. The main advantage of a model selection's approach is that the resulting CS are non asymptotic, i.e. (\ref{eq:couv}) holds for all $n$. Moreover, the CS behaves well even if $s$ does not belong to $S$, which outperforms, in that case, the result of Li \cite{Li89}.\\
Let $S$ be a linear space with dimension $d$ and let $(S_m)_{m\in\M_n}$ be a collection of linear subspaces of $S$, with respective dimensions $(d_m)_{m\in\M_n}$. The diameter of our CS on $S$ is upper bounded, for any $s$ in $S_m$, by $C(\sqrt{d}\vee d_m)/n$, where $C$ is a constant, free from $d$, $d_m$, and $n$. This bound is optimal in the minimax sense. Hence, adaptation is possible over collections of subspaces with dimension $d_m\geq \sqrt{d}$ for $L^2$-balls. This positive result does not hold in general, in particular, adaptation is impossible for $L^{\infty}$-balls (Low \cite{Lo97}). However, the adaptation property is strongly limited since it is impossible over spaces with dimension $d_m\leq \sqrt{d}$. This negative result was already proved asymptotically in  Li \cite{Li89}, Hoffmann $\&$ Lepski \cite{HL02}, Juditski $\&$ Lambert-Lacroix \cite{JLL03}, Robins $\&$ van der Vaart \cite{RV06}. It was proved non-asymptotically in a regression framework in Baraud \cite{Ba04}. We use the method of Baraud \cite{Ba04} and extend his result to the density estimation framework.\\
The paper is decomposed as follows. Section \ref{Ch5S1} introduces the notations and the main assumptions. Section \ref{Ch5S2} presents the technical tools required for the construction of our CS. Section \ref{Ch5s4} gives the main results, we build our CS, give upper bounds on their size and prove their optimality in the minimax sense. Section \ref{Ch5S4} presents a short simulation study, where we illustrate the behavior of our resampling-based estimators. All the proofs are postponed to Section \ref{Ch5S5}. We add in an Appendix the proofs of some technical lemmas.

\section{Notations and assumptions}\label{Ch5S1}

\subsection{Notations}\label{sub21:Not}
Hereafter, $L^2(\mu)$ denotes the space of all measurable functions $t:\R\rightarrow \R$ such that $\int_{\R}t^2(x)d\mu(x)<\infty$. It is endowed by its classical scalar product defined, for all $t$, $t'$ in $L^2(\mu)$ by $<t,t'>=\int_{\R}t(x)t'(x)d\mu(x)$
and by the associated $L^2$-norm defined, for $t$ in $L^2(\mu)$ by $\|t\|=\sqrt{<t,t>}.$\\
For any density $s$, we denote by $\p_s$ the distribution of an iid sample $X_{1:n}=(X_1,...,X_n)$ with common marginal density $s$ and by $\E_s$ the expectation with respect to $\p_s$. \\
Hereafter, $S$, with various subscripts, denotes a linear subspace of $L^2(\mu)$ and $S^*$ the set of densities in $S$. For all sets $\F$ in $L^2(\mu)$, the $L^2$-diameter of $\F$ is defined by 
$$\Delta(\F)=\sup_{(t,t')\in \F^2}\|t-t'\|.$$
For a random set $\hat{B}$ in $L^2(\mu)$, a linear space $S$ of measurable functions and a real number $\alpha$ in $(0,1)$, we define the $(S,\alpha)$-size of $\hat{B}$ as 
\begin{equation}\label{def:sizeCS}
\Delta_{(S,\alpha)}(\hat{B})=\inf\left\{\delta>0,\;\sup_{s\in S^*}\p_s(\Delta(\hat{B})>\delta)\leq \alpha\right\}.
\end{equation}
For all indexes sets $\Lambda$, $(\psi_{\lambda})_{\lambda\in\Lambda}$ will always denote an orthonormal system in $L^2(\mu)$. 

\subsection{Efron's resampling heuristic}\label{sub:Efrheu}
Let $X,X_1,...,X_n$ be i.i.d random variables with common density $s$, let $P_s$ and $P_n$ denote the following processes defined respectively for all functions $t$ in $L^2(\mu)$ and for all measurable functions $t$ by 
$$P_st=<s,t>=\int_{\R}t(x)s(x)d\mu(x)=\E(t(X)),\;P_nt=\frac1n\sum_{i=1}^nt(X_i).$$
Hereafter, a resampling scheme $(W_1,...,W_n)$ is a vector of real valued random variables, independent of $(X_1,....,X_n)$ and exchangeable, which means that, for all permutations $\tau$ of $1,...,n$,
$$(W_{\tau(1)},...,W_{\tau(n)})\;{\rm has}\; {\rm the}\;{\rm same}\;{\rm law}\;{\rm as}\; (W_1,...,W_n).$$
Let $(W_1,...,W_n)$ be a resampling scheme, let $\bar{W}_n=\sum_{i=1}^nW_i/n$ and let $P_n^W$ denotes the resampling-based empirical process defined, for all measurable functions $t$, by 
$$P_n^Wt=\frac1n\sum_{i=1}^nW_it(X_i).$$ 
For all random variables $F(X_1,...,X_n,W_1,...,W_n)$, we denote by 
$$\E_W\left(F(X_1,...,X_n,W_1,...,W_n)\right)=\E\left(F(X_1,...,X_n,W_1,...,W_n)|X_1,...,X_n\right).$$
Let $F$ be a known functional and $F_n=F(P_n,P_s)$, we define the resampling estimator of $F_n$ by
$$F_n^W=C_W\E_W\left(F(P_n^W,\bar{W}_nP_n)\right),$$
where $C_W$ is a constant depending only on the functional $F$ and the law of the resampling scheme. Efron's heuristics states that $F_n^W$ provides a sharp estimator of $F_n$ when the constant $C_W$ is well chosen.

\subsection{Balls in functional spaces}
Our method is strongly based on empirical process methods, in particular on Talagrand's concentration inequality. This inequality involves some $L^{\infty}$-norms, this is why we introduce the following notations. Let $S$ be a linear space of measurable functions. For any function $t$ in $L^2(\mu)\cap L^\infty(\mu)$, let $\pi_S(t)$ denote its orthogonal projection onto $S$, let $\left\|t\right\|_{\infty}$ be its $L^{\infty}$-norm. For all $C$, $C'$, $\eta$ in $\bar{\R}_+$, for all $t$ in $L^2(\mu)$, let 
\begin{equation}\label{def.balls.1}
B_2(t,C,S)=\{t'\in S,\; \|t'-t\|\leq C\},\;B(S)=B_2(0,1,S)=\left\{t\in S,\; \|t\|\leq 1\right\}.
\end{equation}
\begin{equation}\label{def.balls.2}
B_{2,\infty}(C,C',\eta,S)=\left\{t\in L^2(\mu)\cap L^{\infty}(\mu),\; \|t\|\leq C,\; \left\|t\right\|_{\infty}\leq C',\;\|t-\pi_S(t)\|\leq \eta\right\}.
\end{equation}

\subsection{Basic definitions}
\begin{definition}(Confidence Sets)\label{def:CS}\\
Let $(X_1,...,X_n)$ be an i.i.d. sample of real valued random variables, let $S\subset L^2(\mu)$ and let $\beta$ be a real number in $(0,1)$. The set $CS(S,\beta)$ of $(1-\beta)$-confidence balls on $S$ is defined as the collection of all subsets $\hat{B}_{\beta}=B_2(\hat{s},\hat{\rho}_{\beta},S)$ of $L^2(\mu)$, where $\hat{s}$ and $\hat{\rho}_{\beta}$ are measurable with respect to $\sigma(X_1,...,X_n)$ such that
\begin{equation*}
\forall s\in S^*,\; \p_s\left(s\in \hat{B}_{\beta}\right)\geq 1-\beta.
\end{equation*}
\end{definition}
\begin{definition}(Minimax rate of convergence for confidence sets)\label{def:MinMaxRate}\\
Let $(X_1,...,X_n)$ be an i.i.d. sample of real valued random variables, let $S'\subset S\subset L^2(\mu)$ and let $\alpha$, $\beta$ be real numbers in $(0,1)$. The $(\alpha,\beta)$-minimax rate of convergence over $S'$ for CS on $S$ is defined as
\begin{equation*}
\phi_n(\alpha,\beta,S,S')=\inf_{\hat{B}_{\beta}\in CS(S,\beta)}\Delta_{(S',\alpha)}(\hat{B}_{\beta}).
\end{equation*}
\end{definition}
\begin{definition}(Adaptive confidence sets)\label{def:MinMaxSet}\\
Let $(X_1,...,X_n)$ be an i.i.d. sample of real valued random variables, let $S\subset L^2(\mu)$, let $(S_m)_{m\in\M_n}$ be a collection of subsets of $S$ and let $\alpha$, $\beta$ be real numbers in $(0,1)$. A CS $\hat{B}_{\beta}$ in $CS(S,\beta)$ is said to be optimal, or adaptive over $(S_m)_{m\in\M_n}$, if the following condition holds.\\
For all fixed $\alpha$ in $(0,1)$, there exists a constant $c(\alpha,\beta)>0$ free from $n$, $S$ and $(S_m)_{m\in\M_n}$ such that, for all $m$ in $\M_n$,
\begin{equation*}
\Delta_{S_m,\alpha}(\hat{B}_{\beta})\leq c(\alpha,\beta)\phi_n(\alpha,\beta,S,S_m)
\end{equation*}
\end{definition}

\begin{definition}(Test)\\
Let $(X_1,...,X_n)$ be an i.i.d. sample of real valued random variables. Let $S$ be a family of densities on $\R$. Let $S_0$, $S_1$ be two disjoint subsets in $S$. A test $T$ of the assumption $H_0: s\in S_0$ against the alternative $H_1: s\in S_1$ is a function $T: \R^n\rightarrow \{0,1\}$. The test $T$ is said to have a confidence level $1-\alpha\in(0,1)$ when 
$$\forall s\in S_0,\;\p_s\left(T(X_1,...,X_n)=0\right)\geq 1-\alpha.$$
It is said to have a power $1-\beta\in (0,1)$ when 
$$\forall s\in S_1,\;\p_s\left(T(X_1,...,X_n)=1\right)\geq 1-\beta.$$
\end{definition}

\subsection{Main Assumptions}\label{sub22:Ass}
Let $(S_m)_{m\in\M_n}$ be a collection of linear subspaces of $L^2(\mu)$, with finite dimensions respectively denoted by $(d_m)_{m\in\M_n}$. We make the following assumptions on this collection.\\
{\bf H1}: There exists $m_n$ in $\M_n$ such that $S_{m_n}={\rm Span}\left(\bigcup_{m\in\M_n}S_m\right)$.\\
{\bf H2}: There exists a constant $C_1$ such that, for all $m$ in $\M_n$, for all $t$ in $S_m$ 
$$\left\|t\right\|_{\infty}\leq C_1\sqrt{d_m}\|t\|.$$
The last assumption is only technical and let us simplify the results. Let $\beta$ be a real number in $(0,1)$.\\
{\bf H3$(\M,\beta)$}: For all $n\geq 2$ $N_n={\rm Card}(\M_n)$ is finite and there exists a constant $C_\M$ such that, for all $n\geq 2$,
$$\frac{2\sqrt{d_n}\ln(6N_n/\beta)}n\leq C_\M.$$
Four examples are usually developed as fulfilling this set of assumptions:\\
{\bf{[Hist]}} regular histogram spaces: for all $m$ in $\N^*$, $S_m$ is the space of all the functions constant on the partition $(I_{[k/m,(k+1)/m)})_{k=0,...,m-1}$ of $[0,1]$, $d_m=m$.\\
{\bf{[T]}} trigonometric spaces: $S_m$ is the linear span of the functions $\psi_{0,0}(x)=1_{[0,1]}$, $\psi_{j,1}(x)=\sqrt{2}\cos(2\pi jx)1_{[0,1]}(x)$ and $\psi_{j,2}(x)=\sqrt{2}\sin(2\pi jx)1_{[0,1]}(x)$ for all $1\leq j\leq J_m$. $d_m=2J_m+1$.\\
{\bf{[P]}} regular piecewise polynomial spaces: $S_m$ is the linear span of the functions $(\psi_{j,k})$ for $j=1,...,J_m$, $k=0,...,r-1$, where, for all $j=1,...,J_m$ and $k=0,...,r-1$, $\psi_{j,k}$ is a polynomial of degree $k$ on $[(j-1)/J_m,j/J_m]$.  $d_m=rJ_m$.\\
{\bf{[W]}} spaces spanned by dyadic wavelets with regularity $r$.\\
We have to choose $d_{m_n}\leq Cn^2/(\ln n)^2$ and $\beta\geq n^{-r}$ for some $r>0$ in order to fulfill Assumption {\bf H3$(\M,\beta)$}.
For a description of those spaces and their properties, we refer to Birg\'e $\&$ Massart \cite{BM97}. Hereafter, in order to simplify the notations, we will often write $S_n$, $d_n$, $s_n$,... instead of $S_{m_n}$, $d_{m_n}$, $s_{m_n}$,...

\section{Technical tools}\label{Ch5S2}
This section presents the results required in Section \ref{Ch5s4} to build our adaptive confidence sets. Let $s$ be a density in $L^2(\mu)$ and let $s_m$ and $s_n$ denote respectively its orthogonal projections onto the linear spaces $S_m$ and $S_n$, where $S_m\subset S_n$. We recall the definition and some basic properties of the projection estimator $\hat{s}_m$ of $s$ on $S_m$ in Section \ref{sub22:projest}. From Pythagoras theorem, it satisfies
\begin{equation}\label{eq:pyth}
\|s-\hat{s}_m\|^2=\|s-s_n\|^2+\|s_n-s_m\|^2+\|s_m-\hat{s}_m\|^2.
\end{equation}
Section \ref{sub23:estvar} deals with the estimation of $\|s_m-\hat{s}_m\|^2$. We introduce our resampling estimator and state a very important concentration inequality (Theorem \ref{theo:CpW}). In Section \ref{sub24:estbiais}, we introduce our estimator of $\|s_n-s_m\|^2$ based on $U$-statistics.
\subsection{Projection estimators}\label{sub22:projest}
\begin{definition}(projection estimators)\label{def:sm}\\
Let $X_1,...,X_n$ be i.i.d random variables with common density $s$ in $L^2(\mu)$. Let $S_m$ be a linear subspace of $L^2(\mu)$. The projection estimator of $s$ on $S_m$ is defined by
$$\hat s_m=\inf_{t\in S_m}\|t\|^2-2P_nt.$$
\end{definition}
Classical computations show the following Lemma:
\begin{lemma}
Let $X_1,...,X_n$ be i.i.d random variables with common density $s$ in $L^2(\mu)$. Let $S_m$ be a linear subspace of $L^2(\mu)$ and let $(\psi_{\lambda})_{\lambda\in \Lambda_m}$ be an orthonormal basis of $S_m$. Let $s_m$ be the orthogonal projection of $s$ onto $S_m$ and let $\hat{s}_m$ be the projection estimator of $s$ onto $S_m$. Then,
$$s_m=\sum_{\lambda\in \Lambda_m}(P_s\psi_{\lambda})\psi_{\lambda},\; \hat{s}_m=\sum_{\lambda\in \Lambda_m}(P_n\psi_{\lambda})\psi_{\lambda},\;\|s_m-\hat{s}_m\|^2=\sum_{\lambda\in \Lambda_m}\left[(P_n-P_s)\psi_{\lambda}\right]^2.$$
\end{lemma}

\subsection{Estimation of $\|s_m-\hat{s}_m\|^2$ by resampling methods}\label{sub23:estvar}
Let $s$ be a density in $L^2(\mu)$. Let $S_m$ be a finite dimensional linear subspace of $L^2(\mu)$, let $(\psi_{\lambda})_{\lambda\in \Lambda_m}$ be an orthonormal basis of $S_m$. Let $s_m$ denote the orthogonal projection of $s$ onto $S_m$ and let $\hat{s}_m$ denote the projection estimator of $s$ onto $S_m$. $\|s_m-\hat{s}_m\|^2$ is a functional of $P_n$ and $P_s$, therefore, it can be estimated by resampling. Indeed, let $(W_1,...W_n)$ be a resampling scheme and let $\bar{W}_n=\sum_{i=1}^nW_i/n$. The resampling estimator of $\|s_m-\hat{s}_m\|^2$ given by Efron's heuristic (see Section \ref{sub:Efrheu}) is defined for this resampling scheme and a suitably chosen constant $C_W$ by:
\begin{equation}\label{def:pWm}
p_W(S_m)=C_W\sum_{\lambda\in\Lambda_m}\E_W\left([(P_n^W-\bar{W}_nP_n)\psi_{\lambda}]^2\right).
\end{equation}
$p_W(S_m)$ is well defined since we can check with Cauchy-Schwarz inequality that
$$p_W(S_m)=C_W\E_W\left(\left[\sup_{t\in S_m, \|t\|\leq 1}(P_n^W-\bar{W}_nP_n)t\right]^2\right).$$
The deviations of $p_W(S_m)$ are given by the following theorem.
\begin{theo}\label{theo:CpW}
Let $S_m$ be a linear subspace of $L^2(\mu)$ with finite dimension $d_m$, satisfying {\bf H2} and let $C_3>0$. Let $X_1,...,X_n$ be an i.i.d. sample, let $(W_1,...W_n)$ be a resampling scheme and let $p_W(S_m)$ be the associated random variables defined in (\ref{def:pWm}) for $C_W=\left(\Var(W_1-W_n)\right)^{-1}$. There exists a constant $\kappa_v(C_1,C_3)$ such that, for all $2\leq x\leq C_3n/\sqrt{d_m}$, for all densities $s$ in $L^2(\mu)\cap L^{\infty}(\mu)$,
$$\p_s\left(\|s_m-\hat{s}_m\|^2> p_W(S_m)+\kappa_v(C_1,C_3)(1+\sqrt{\left\|s\right\|_{\infty}\wedge\left\|s\right\|d_m^{1/2}\wedge d_m})\frac{\sqrt{d_m} x}n\right)\leq e^{-x/2}.$$
\end{theo}
{\bf Comments:} 
\begin{itemize}
\item This theorem is one of the main contributions of the article. It provides a sharp control of the variance term. It is the main difference with the article of Baraud who worked in a Gaussian framework and handled this term with a concentration inequality for $\chi^2$-statistics of Birg\'e \cite{Bi02}. Our new construction is more general and can be easily adapted to other frameworks, which is not the case in Baraud \cite{Ba04}.
\item It is proved thanks to a technical lemma (Lemma \ref{lem:ps-pW=U}) and a sharp concentration inequality (Lemma \ref{lem:concU}). Lemma \ref{lem:ps-pW=U} shows that, with our choice of $C_W$, $\|s_m-\hat{s}_m\|^2-p_W(S_m)$ is a totally degenerate $U$-statistics of order 2. Lemma \ref{lem:concU} is a concentration inequality for $U$-statistics of order 2.
\item The proof of Lemma \ref{lem:concU} is derived from Houdr\'e $\&$ Reynaud-Bouret \cite{HRB03}, it follows mainly the one of Fromont $\&$ Laurent \cite{FL06}. The main improvement compared with Fromont $\&$ Laurent \cite{FL06} is that we work with general linear spaces $S_{m}$.
\item The bound involves a term $\sqrt{\left\|s\right\|_{\infty}}\wedge\sqrt{\left\|s\right\|}d_m^{1/4}\wedge \sqrt{d_m}$. From a theoretical point of view, the term $\sqrt{\left\|s\right\|}d_m^{1/4}\wedge \sqrt{d_m}$ is useless asymptotically when $\left\|s\right\|_{\infty}$ is finite. In practice the $L^2$-norm of $s$ is often much smaller than its $L^{\infty}$-norm. Moreover, our control can also be used when $\left\|s\right\|_{\infty}$, $\left\|s\right\|$ or both of these quantities are unknown, since $\kappa_v(C_1,C_3)$ is free from $\left\|s\right\|$, $\left\|s\right\|_{\infty}$.
\item The condition on $x$ is not a problem in practice. We are interested in cases where $1-e^{-x/2}$ is large, therefore, $2\leq x$ will always be satisfied. Moreover, we will see in Section \ref{Ch5s4} that the assumptions {\bf H3$(\M,\beta)$} are designed to ensure that the interesting $x$ satisfy $x\leq C_3n/\sqrt{d_m}$ provided that $C_3$ is sufficiently large.
\item This theorem can be used to build a model selection procedure of density estimation. Actually, an ideal penalty in this problem is given by $2\|s_m-\hat{s}_m\|^2$ and the aim of model selection is to evaluate this ideal penalty as precisely as possible. Theorem \ref{theo:CpW} provides such a control. This important application is discussed in detail in \cite{Le09}. For an introduction to model selection, we refer to Massart \cite{Ma07}. The concept of ideal penalty is defined in Arlot \cite{Ar08}.
\item In order to keep the result as readable as possible, we only give the explicit form of the constant $\kappa_v(C_1,C_3)$ in the proof of Theorem \ref{theo:CpW}.
\end{itemize}

\begin{coro}\label{cor:Vm}
Let $X_1,...,X_n$ be i.i.d. real valued random variables. Let $(S_m)_{m\in\M_n}$ be a collection of finite dimensional linear spaces satisfying {\bf H1}, {\bf H2}. Let $\beta\in (0,1)$ such that this collection satisfies also {\bf H3$(\M,\beta)$} and let $M_2>0$, $M_{\infty}>0$. Let $(W_1,...,W_n)$ be a resampling scheme and let $p_W(S_m)$ be the associated resampling estimator defined in Theorem \ref{theo:CpW}. Let $\kappa_v(C_1,C_\M)$ be the constant defined in Theorem \ref{theo:CpW} for $C_3=C_\M$, let $x_n=2\ln\left(2N_n/\beta\right)\vee 2$ and let
\begin{equation}\label{def:Vm}
V(m,\beta,X_1,...,X_n)=p_W(S_m)+\kappa_v(C_1,C_\M)\left(1+\sqrt{M_{\infty}\wedge M_2d_m^{1/2}\wedge d_m}\right)\frac{\sqrt{d_m}x_n}n
\end{equation}
Then, for all densities $s$ in $L^2(\mu)\cap L^{\infty}(\mu)$ such that $\|s\|\leq M_2$ and $\left\|s\right\|_\infty\leq M_\infty,$
$$\p_s\left(\exists m\in \M_n,\;\|s_m-\hat{s}_m\|^2> V(m,\beta,X_1,...,X_n)\right)\leq \frac{\beta}2.$$
\end{coro}
{\bf Comments:}
\begin{itemize}
\item This corollary gives a uniform upper bound $V(m,\beta,X_1,...X_n)$ on the variance term.
\item The size of this uniform bound, in the sense of (\ref{def:sizeCS}), is given by the following Theorem.
\end{itemize}
\begin{theo}\label{theo:SizeCB}
Let $X_1,...,X_n$ be i.i.d. real valued random variables. Let $(S_m)_{m\in \M_n}$ be a collection of linear spaces satisfying {\bf H1}, {\bf H2}. Let $\alpha$, $\beta$ be real numbers in $(0,1)$ such that this collection satisfies also {\bf H3$(\M,\alpha)$} and {\bf H3$(\M,\beta)$}. Let $M_2>0$, $M_{\infty}>0$ and let $V_{m,\beta}=V(m,\beta,X_1,...,X_n)$ be the associated random variables defined in (\ref{def:Vm}). There exists a constant $\kappa$, free from $d_m$, $M_2$, $M_{\infty}$, $\alpha$, $\beta$, such that, for all $m$ in $\M_n$,
$$\Delta^2_{B_{2,\infty}(M_2,M_\infty,0,L^2(\mu)),\alpha}\left(V_{m,\beta}\right)\leq \kappa\left[\frac{d_m}n+\left(1+\sqrt{M_{\infty}\wedge M_2d_m^{1/2}\wedge d_m}\right)\frac{\sqrt{d_m}}n\ln\left[ \frac{N_n}{\alpha\beta}\right]\right].$$
\end{theo}
{\bf Comments:}
\begin{itemize}
\item For fixed confidence level $\alpha$, $\beta$, the asymptotic order of magnitude of $V_{m,\beta}$ is $d_m/n$ for all models with dimension $d_m\geq (\ln N_n)^2$.
\end{itemize}

\subsection{Estimation of $\|s_n-s_m\|^2$}\label{sub24:estbiais}
The simple following lemma is important to understand our procedure.
\begin{lemma}
Let $X_1,...,X_n$ be i.i.d. real valued random variables with common density $s$ in $L^2(\mu)$. Let $S_m\subset S_n$ be two linear subspaces of $L^2(\mu)$, with respective finite dimensions $d_m$ and $d_n$. Let $s_m$ and $s_n$ be the orthogonal projections of $s$ respectively onto $S_m$ and $S_n$.
Let $(\psi_{\lambda})_{\lambda\in\Lambda_n}$ be an orthonormal basis of $S_n$ such that $(\psi_{\lambda})_{\lambda\in\Lambda_m}$ is an orthonormal basis of $S_m$, with $\Lambda_m\subset \Lambda_n$. Then
\begin{equation}\label{eq:biais}
\|s_n-s_m\|^2=\sum_{\lambda\in\Lambda_n-\Lambda_m}(P_s\psi_{\lambda})^2=\E_{s}\left(\frac1{n(n-1)}\sum_{i\neq j=1}^n\sum_{\lambda\in \Lambda_n-\Lambda_m}\psi_{\lambda}(X_i)\psi_{\lambda}(X_j)\right)
\end{equation}
\end{lemma}
Based on this kind of lemma, Laurent \cite{La96, La05} introduced the estimators based on $U$-statistics to estimate quadratic functionals of a density. These estimators were successfully used by Fromont $\&$ Laurent \cite{FL06} for goodness of fit tests in a density estimation model, and by Robins $\&$ van der Vaart \cite{RV06} to build adaptive confidence sets. We follow the same steps here and define, for any observation $X_1,...X_n$, for all finite dimensional linear spaces $S_m\subset S_n$, for all orthonormal basis $(\psi_{\lambda})_{\lambda\in\Lambda_n}$ of $S_n$ such that $(\psi_{\lambda})_{\lambda\in\Lambda_m}$ is an orthonormal basis of $S_m$, with $\Lambda_m\subset \Lambda_n$,
\begin{equation}\label{def:pb}
p_b(S_m,S_n)=\frac1{n(n-1)}\sum_{i\neq j=1}^n\sum_{\lambda\in \Lambda_n-\Lambda_m}\psi_{\lambda}(X_i)\psi_{\lambda}(X_j).
\end{equation}
$p_b(S_m,S_n)$ is well defined since we can prove with Cauchy-Schwarz inequality that, if $S_n^{\bot m}$ denotes the orthogonal of $S_m$ in $S_n$,
$$p_b(S_m,S_n)=\frac1{n-1}\left(n\sup_{t\in B_2(S_n^{\bot m})}(P_nt)^2-P_n\left(\sup_{t\in  B_2(S_n^{\bot m})}t^2\right)\right).$$
The deviations of $p_b(S_m,S_n)$ are given by the following result:
\begin{lemma}\label{lem:concbiais}
Let $X_1,...,X_n$ be i.i.d. real valued random variables. Let $S_m\subset S_n$ be two linear subspaces of $L^2(\mu)$, with respective finite dimensions $d_m$ and $d_n$ and let  $p_b(S_m,S_n)$ be the estimator defined in (\ref{def:pb}). For any density $s$ in $L^2(\mu)$, let $s_n$ and $s_m$ denote its orthogonal projections respectively onto $S_n$ and $S_m$. For all $C_3>0$ and all $\epsilon$ in $(0,1)$, there exists a real constant $\kappa_b(\epsilon,C_3)$ such that, for all $2\leq x\leq C_3n/\sqrt{d_n}$, for all densities $s$ in $L^2(\mu)\cap L^{\infty}(\mu),$ with $\p_s$-probability larger than $1-3e^{-x/2}$,
$$\left|p_b(S_m,S_n)-\|s_n-s_m\|^2\right|\leq \epsilon\|s_n-s_m\|^2+\kappa_b(\epsilon,C_3)\left(1+\sqrt{\left\|s\right\|_{\infty}\wedge \left\|s\right\|_2d_n^{1/2}}\right)\frac{\sqrt{d_n}x}n.$$
\end{lemma}
Thanks to this Lemma, we can derive the following corollary that gives our estimation of $\left\|s_n-s_m\right\|$.

\begin{coro}\label{cor:estbiais}
Let $X_1,...,X_n$ be i.i.d. real valued random variables. Let $(S_m)_{m\in\M_n}$ be a collection of linear spaces satisfying assumptions {\bf H1}, {\bf H2}. Let $\beta$ be a real number in $(0,1)$ such that this collection satisfies also {\bf H3$(\M,\beta)$}. Let $M_2>0$, $M_{\infty}>0$, $x_n=2\ln\left(6N_n/\beta\right)\vee 2$. Let $p_b$ be defined in (\ref{def:pb}) and, for all $\epsilon$ in $(0,1)$, let $\kappa_b(\epsilon,C_\M)$ be the constant defined in Lemma \ref{lem:concbiais} for $C_3=C_\M$. For all $m\in \M_n$, let
\begin{equation}\label{def:Bm}
K(m,\beta,X_1,...,X_n)=\inf_{\epsilon\in(0,1)}\frac{p_b(S_m,S_n)}{1-\epsilon}+\frac{\kappa_b(\epsilon,C_\M)}{1-\epsilon}\left(1+\sqrt{M_{\infty}\wedge M_2d_n^{1/2}}\right)\frac{\sqrt{d_n}x_n}n.
\end{equation}
Then, for all densities $s$ in $B_{2,\infty}(M_2,M_\infty,0,L^2(\mu))$,
\begin{equation*}\label{eq:estbiais}
\p_s\left(\exists m\in\M_n,\;\left\|s_n-s_m\right\|^2>K(m,\beta,X_1,...,X_n)\right)\leq \frac{\beta}2.
\end{equation*}
\end{coro}
{\bf Comments:}
\begin{itemize}
\item This corollary gives a sharp estimation of the bias term. In particular, we will see in the following section that the term $\sqrt{d_n}x_n/n$ is essentially necessary.
\item We obtain a bound valid for all the models in the collection $\M_n$. Combined with Corollary \ref{cor:Vm}, it gives all the tools required to apply our method of selection.
\end{itemize}

\section{Main results}\label{Ch5s4}
\subsection{Adaptive Confidence Balls}\label{sec:ACB}
We can now easily present our model selection procedure to obtain CS.
\subsubsection*{Construction of the adaptive CS}
Let $\beta$ be a real number in $(0,1)$, let $M_2>0$, $M_{\infty}>0$, let $(S_m)_{m\in\M_n}$ be a collection of finite dimensional linear spaces and let $S_n={\rm Span}\left(\bigcup_{m\in\M_n}S_m\right)$. Let $(V(m,\beta,X_1,...,X_n))_{m\in\M_n}$ be the collection defined in (\ref{def:Vm}), let $(K(m,\beta,X_1,...,X_n))_{m\in\M_n}$ be the collection defined in (\ref{def:Bm}) and let $\eta$ be a positive real number.
For all $m$ in $\M_n$, let
\begin{equation*}\label{def:hatrho}
\hat{\rho}(m,\eta,\beta)=\sqrt{\eta^2+K(m,\beta,X_1,...,X_n)+V(m,\beta,X_1,...,X_n)}.
\end{equation*}
Recall the definition of the $L^2$-ball centered in an element $t$ of $L^2(\mu)$ with radius $C$ in $\R$ given in (\ref{def.balls.1}). Our final CS is defined by
\begin{equation}\label{def:hatm}
\hat{B}_{\beta,\eta}=B_2(\hat{s}_{\hat{m}},\hat{\rho}(\hat{m},\eta,\beta),L^2(\mu)),\;{\rm where}\;\hat{m}=\arg\min_{m\in\M_n}\left\{\hat{\rho}(m,\eta,\beta)\right\}.
\end{equation}
\subsubsection*{Performances of our CS}
\begin{theo}\label{theo.ACS}
Let $X_1,...,X_n$ be i.i.d real valued random variables. Let $(S_m)_{m\in\M_n}$ be a collection of models satisfying assumptions {\bf H1}, {\bf H2}. Let $\beta$ be a real number in $(0,1)$ such that this collection satisfies also {\bf H3$(\M,\beta)$}. Let $M_2>0$, $M_{\infty}>0$, $\eta>0$ and let $B_{2,\infty}(M_2,M_{\infty},\eta,S_n)$ be the ball defined in (\ref{def.balls.2}).\\
Then $\hat{B}_{\beta,\eta}$, defined in (\ref{def:hatm}), belongs to $CS(B_{2,\infty}(M_2,M_{\infty},\eta,S_n),\beta)$.\\
Moreover, there exists a constant $\kappa$ such that for all $m$ in $\M_n$, for all $\eta_m>0$ and all $\alpha$ such that $(S_m)_{m\in\M_n}$ satisfies also {\bf H3$(\M,\alpha)$}
\begin{equation}\label{eq:sizeCS}
\Delta_{B_{2,\infty}(M_2,M_\infty,\eta_m,S_m),\alpha}(\hat{B}_{\beta,\eta})\leq \kappa\left((\eta_m^2+ \frac{d_m}n)\vee (\eta^2+\frac{\sqrt{d_n}\ln(N_n/(\alpha\beta))}n)\right).
\end{equation}
\end{theo}
{\bf Comments:}
\begin{itemize}
\item Theorem \ref{theo.ACS} gives CS over $B_{2,\infty}(M_2,M_{\infty},\eta,S_n)$, with prescribed confidence level $\beta$, valid for all $n\geq 2$.
\item The size of these CS is upper bounded by the maximum of two terms. $\eta^2+\sqrt{d_n}/n$ is the minimax separation rate for the tests $H_0:\;s=s_0$ against the alternative $H_1:\; s\in B_{2,\infty}(M_2,M_{\infty},\eta,S_n)-\{s_0\}$, where $s_0$ is some element in $S_m^*$. $\eta_m^2+ d_m/n$ is the minimax estimation rate over $B_{2,\infty}(M_2,M_\infty,\eta_m,S_m)$.
\item Robins $\&$ van der Vaart \cite{RV06} proved that these rates are optimal asymptotically. We will show in Theorem \ref{theo.MinMaxlb} below that this property holds also non asymptotically. 
\item $\hat{\rho}(m,\eta,\beta)$ has basically the following form
\[
\hat{\rho}^2(m,\eta,\beta)=\eta^2+p_b(S_m,S_n)+p_W(S_m)+\kappa(M_2,M_\infty)\frac{\sqrt{d_n}\ln(N_n/(\alpha\beta))}n.
\]
It depends in practice on two unknown constants, $\eta$ and $\kappa(M_2,M_\infty)$. We believe that some "slope heuristic" (see Birg\'e $\&$ Massart \cite{BM07}, Arlot $\&$ Massart \cite{AM08} or \cite{Le09}) method can be developed for CS in order to obtain a data driven estimate of $\kappa(M_2,M_\infty)$. This estimate would probably be more reasonable than the upper bound given in our proof. On the other hand, we believe that the constant $\eta$ can only be handled with suitably chosen assumptions. For example, some regularity assumption as in Section \ref{ss34} bellow.
\item Baraud \cite{Ba04} used a procedure almost similar in a regression framework. He defined, for all $m$ in $\M_n$, a test $T_m$ to test the null hypothesis $s_n\in S_m$ against the alternative $s_n\in S_n-S_m$ and some positive number $\hat{\rho}(m)$. His $\hat{\rho}(m)$'s are calibrated to satisfy the property that, if $T_m$ accepts the null, then, with probability close to one, the distance between $s$ and its projection estimator $\hat{s}_m$ is not larger than $\hat{\rho}(m)$. He selected $\hat{m}$ as the minimizer of $\hat{\rho}(m)$ among those $m$ for which $T_m$ accepts the null and defined the confidence ball as the $L^2$-ball centered at $\hat{s}_{\hat{m}}$ of radius  $\hat{\rho}(\hat{m})$. The main difference with this general scheme is that our procedure does not require a series of tests to work as the bound given in Corollary \ref{eq:estbiais} holds for all $m$.
\end{itemize}
\subsection{Optimality of our balls}\label{Ch5S3}
In this section we prove that the rate given in (\ref{eq:sizeCS}) can not be improved in general, from a minimax point of view. The result is stated in the following theorem:
\begin{theo}\label{theo.MinMaxlb}
Let $S_{n}$ be the set of histograms on $\{[k/d_{n},(k+1)/d_{n}),\; k=0,...,d_{n}-1\}$ and let $S_m$ be the linear subspace of $S_{n}$ of histograms on $\{[k/d_m,(k+1)/d_m),\;k=0,...,d_m-1\}$. Let $\alpha,\beta$ be real numbers in $(0,1)$ such that $2\alpha+\beta<1$.  There exists a constant $C(\alpha,\beta)$, such that
$$\phi^2_n(\alpha,\beta,S_n,S_m)\geq C(\alpha,\beta)\left(\frac{\sqrt{d_n}}n\vee \frac{d_m}n\right).$$
\end{theo}
{\bf Comments:}
\begin{itemize}
\item Theorem \ref{theo.MinMaxlb} gives the optimality of the rate given in (\ref{eq:sizeCS}), since the terms $\eta$ and $\eta_m$ can obviously not be avoided also.
\item The key point of the proof (Lemma \ref{5}) is that we can not build a test of null hypothesis $H_{0}: s\in S_{m}$ against the alternative $H_{1}: s\in S_{n},\;s\notin S_{m}$ with separation rate smaller than $C_{\alpha,\beta}\sqrt{d_{n}}/n$. This extends the result of Ingster \cite{In1,In2,In3} to a non asymptotical framework and the result of Baraud \cite{Ba04} to density estimation. For a definition of the separation rate, we refer to Ingster \cite{In1,In2,In3}.
\item The proof follows the methodology described in Baraud \cite{Ba04}. 
\end{itemize}

\subsection{Application to regular density}\label{ss34}
This section presents the application of Theorem \ref{theo.ACS} to regular densities. In particular, we extend the result of Robins $\&$ van der Vaart \cite{RV06} since (\ref{eq:couv}) is obtained for all $n$. 

\vspace{0.2cm}

\noindent
{\bf Fourier spaces:}\\
For all $k$ in $\N^{*}$, for all $x$ in $\R$, let 
$$\psi_{1,k}(x)=\sqrt{2}\cos(2\pi k x)I_{[0,1]}(x),\; \psi_{2,k}(x)=\sqrt{2}\sin(2\pi k x)I_{[0,1]}(x).$$ 
For all $d$ in $\N$, let $F_d$ be the linear space spanned by the functions $I_{[0,1]},\;\psi_{1,k},\;\psi_{2,k}$, for all $k$ in $\{1,...,d\}$. $F_d$ is a subspace of $L^{2}(\mu)$. It is a classical result (see for example Birg\'e $\&$ Massart \cite{BM97}) that any sub-collection of $(F_{d_{m}})_{0\leq d_{m}\leq n^2(\ln n)^{-2}}$ satisfies {\bf H1, H2} with $C_{1}=1$. We can also easily check that, for all $\beta\geq n^{-2}$, it satisfies also {\bf H3$(\M,\beta)$} with $C_{\M}=4$.

\vspace{0.2cm}

\noindent
{\bf Sobolev Spaces:}\\
For all functions $t$ in $L^{2}(\mu)$, let 
$$t_{0}=\int_{\R}t(x)I_{[0,1]}(x)d\mu(x)=\int_{0}^{1}t(x)d\mu(x)$$ 
and for all $k\in \N^{*}$, let 
$$t_{1,k}=\int_{\R}t(x)\psi_{1,k}(x)d\mu(x),\; t_{2,k}=\int_{\R}t(x)\psi_{2,k}(x)d\mu(x).$$ 
For all $\gamma\in \R_{+}^{*}$, for all $M$ in $\R_{+}$, we denote by $S(\gamma,M)$, the set of functions $t$ in $L^{2}(\mu)$ such that 
$$t_0^2+\sum_{i\in\N^*}\left(t_{1,i}^{2}+t_{2,i}^{2}\right)i^{2\gamma}\leq M^{2}.$$
It is clear that for all $t$ in $S(\gamma,M)$, $\|t\|\leq M$ and for all $d$ in $\N$, if $\pi_{F_d}(t)$ denotes the orthogonal projection of $t$ onto $F_d$,
$$\|t-\pi_{F_d}(t)\|^{2}=\sum_{i>d}\left(t_{1,i}^{2}+t_{2,i}^{2}\right)\leq \frac1{(d+1)^{2\gamma}}\sum_{i>d}\left(t_{1,i}^{2}+t_{2,i}^{2}\right)i^{2\gamma}\leq \frac{M^{2}}{(d+1)^{2\gamma}}.$$
We can also use Cauchy-Schwarz inequality to prove that, when $\gamma>1/2$, for all $x$ in $[0,1]$,
\begin{eqnarray*}
|t(x)|\leq|t_0|+\sqrt{2\left(\sum_{i\in\N}(t^{2}_{1,i}+t_{2,i}^{2})^{2}(i+1)^{2\gamma}\right)\left(\sum_{i\in\N}\frac{\cos^{2}(2\pi ix)+\sin^{2}(2\pi i x)}{(i+1)^{2\gamma}}\right)}.
\end{eqnarray*}
Hence, when $\gamma>1/2$, for all $t$ in $S(\gamma,M)$, $\left\|t\right\|_{\infty}\leq2M\sqrt{\sum_{i\in \N}(i+1)^{-2\gamma}}$.
When $\gamma>1/2$, let $M_{\infty}=2M\sqrt{\sum_{i\in \N}(i+1)^{-2\gamma}}$ and when $\gamma\leq1/2$, let $M_{\infty}$ denote a positive real number. We have obtained that
\begin{equation}\label{eq.SobSpaces}
S(\gamma,M,M_{\infty}):=\{t\in S(\gamma,M),\;\left\|t\right\|_{\infty}\leq M_{\infty}\}\subset B_{2,\infty}\left(M,M_{\infty},M(d+1)^{-\gamma},F_d\right).
\end{equation}
Hence, the following proposition holds.
\begin{prop}
We keep the previous notations. Let $\gamma$, $M$, $M_{\infty}$ be strictly positive real numbers, let $d_{n}$ denotes the integer part of $n^{(2\gamma+1/2)^{-1}}\wedge n^2(\ln n)^{-2}$ and let $\M_n=\{1,...,d_n\}$.\\
Let $\hat{B}_{\beta,M(d_{n}+1)^{-\gamma}}$ be the set defined in Theorem \ref{theo.ACS} for the collection $(F_{d_{m}})_{d_{m}\in\M_n}$. Then, $\hat{B}_{\beta,M(d_{n}+1)^{-\gamma}}$ belongs to $CS(S(\gamma,M,M_{\infty}),\beta)$.\\
There exists a constant $\kappa$ free from $n$ such that, for all $\gamma'\geq \gamma$,
$$\Delta_{S(\gamma',M,M_{\infty}),\alpha}\left(\hat{B}_{\beta,M(d_{n}+1)^{-\gamma}}\right)\leq \kappa\left(n^{-\gamma'/(2\gamma'+1)}\vee (\ln n)n^{-2\gamma/(4\gamma+1)}\right).$$
\end{prop}
{\bf Comments:}
\begin{itemize}
\item This result can be compared with the one of Robins $\&$ van der Vaart \cite{RV06}. Our balls satisfy the covering property (\ref{eq:couv}) for all $n$ and not asymptotically as in their paper. They proved that the rate $n^{-\gamma'/(2\gamma'+1)}\vee n^{-2\gamma/(4\gamma+1)}$ is asymptotically optimal.
\item It is a straightforward consequence of Theorem \ref{theo.ACS}, applied with $\eta_m=M(d_m+1)^{-\gamma'}$, $\eta=M(d_{n}+1)^{-\gamma}$ and the previous computations, therefore, the proof is omitted.
\end{itemize}

\section{Simulation study.}\label{Ch5S4}
In this section, our first goal is to illustrate Theorem \ref{theo:CpW}. We proved that the difference $\left\|s_m-\hat{s}_m\right\|_2^2-p_W(S_m)$ is upper bounded by $\sqrt{d_m}/n$, we will show that this bound is sharp on some simulations. Then, we will consider a more general version of Efron's heuristics, which states that, for a good choice of the constant $C_W$, the distribution of $\left\|s_m-\hat{s}_m\right\|_2^2$ is close to the conditional distribution $\mathcal{D}^W\left(C_W\sum_{\lambda\in \Lambda_m}[(P^W_n-\bar{W}_n)\psi_{\lambda}]^2\right)$. The quantiles of $\left\|s_m-\hat{s}_m\right\|_2^2$ must then be close to their resampled counterpart. In a second simulation, we test this method and remark that it gives very good practical results. 

\subsection{Illustration of Theorem \ref{theo:CpW}}
In this simulation, $s$ is the uniform density on $[0,1]$, $S_m$ is the set of histograms on the partition $([(k-1)/d_m,k/d_m))_{k=1,...,d_m}$.  $(W_1,...,W_n)$ are Efron's weights, i.e. the distribution $\mathcal{D}(W_1,...,W_n)$ is the multinomial distribution $\mathcal{M}(n,1/n,...,1/n)$. In order to compute $p_W(S_m)$, we estimate the conditional expectation $\E^W(\sum_{\lambda\in \Lambda}[(P^W_n-\bar{W}_n)\psi_{\lambda}]^2)$ by a Monte Carlo method with $n_b$ repetitions. Finally, we repeat $p=1000$ times the experiment. We plot the histograms of the $p$ values of the normalized difference $n(\left\|s_m-\hat{s}_m\right\|_2^2-p_W(S_m))/\sqrt{d_m}$. The first histogram is obtained with $n=50, d_m=10, n_b=100$ and the second for $n=200, d_m=50, n_b=500$.

\vspace{0.5cm}

\begin{figure}[htbp]
\begin{center}
\includegraphics[width=7.5cm,height=7cm]{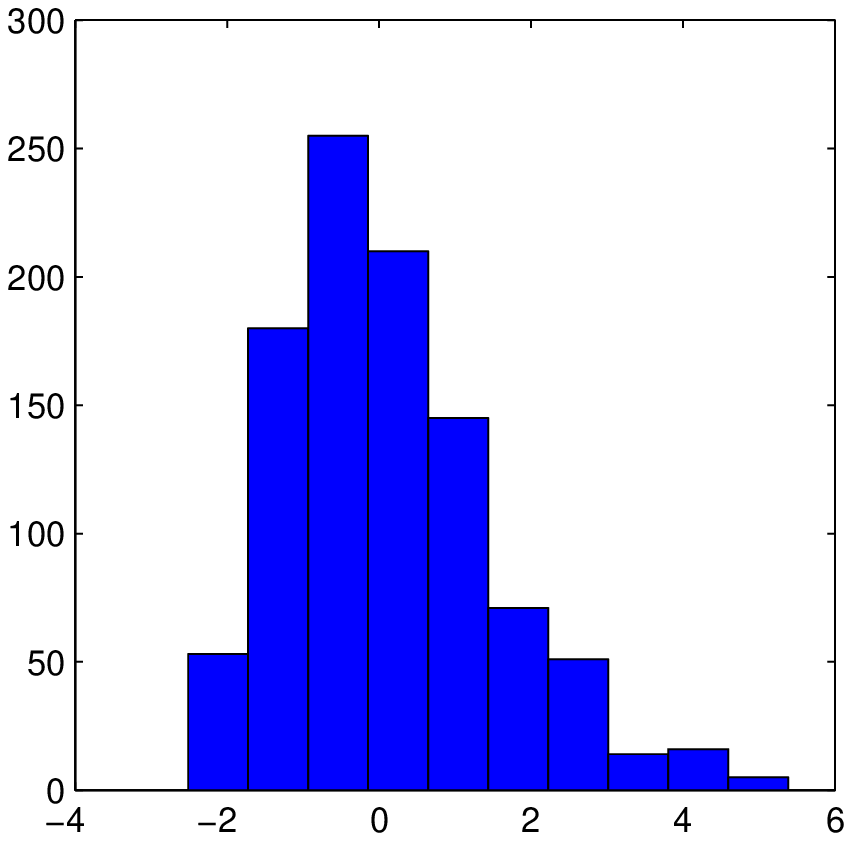}
\includegraphics[width=7.5cm,height=7cm]{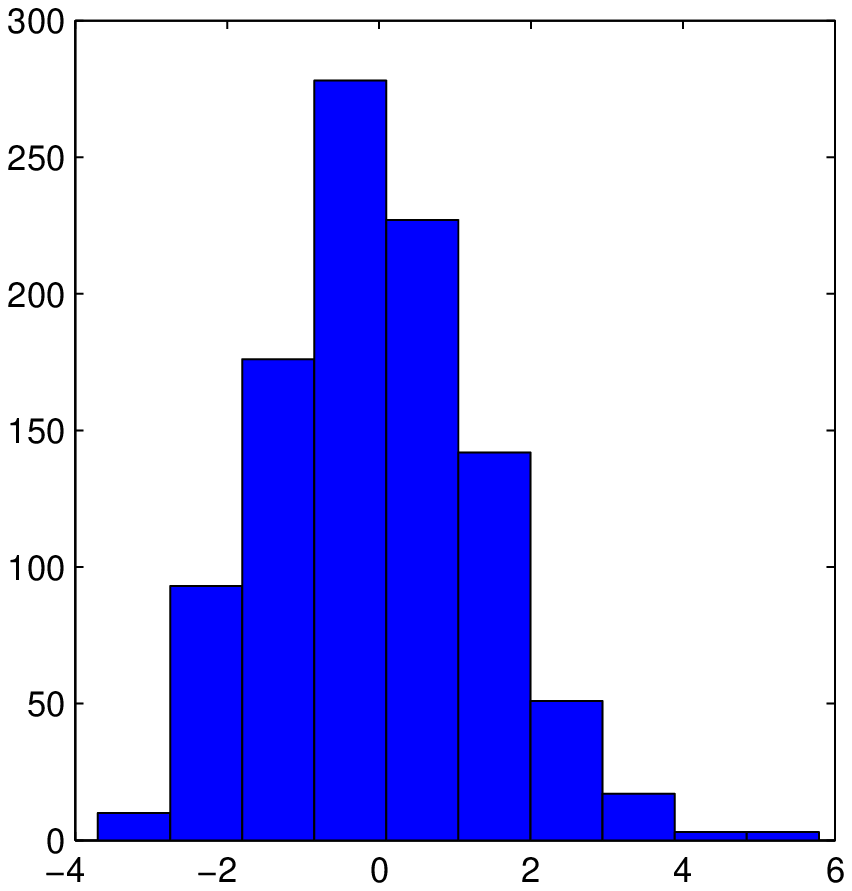}
\caption {$\frac{n}{\sqrt{d_m}}(\left\|s_m-\hat{s}_m\right\|_2^2-p_W(S_m))$.}
\end{center}
\end{figure}
\noindent
{\bf Comments:}
\begin{itemize}
\item The distribution of $n(\left\|s_m-\hat{s}_m\right\|_2^2-p_W(S_m))/\sqrt{d_m}$ does not change with $n$ or $d_m$. This shows that the result of Theorem \ref{theo:CpW} is sharp in this example, at least, up to the constant in front of the remainder term.
\end{itemize}
\subsection{Illustration of the second Efron's heuristic}
In this simulation, we keep the same $s$ and the same resampling scheme. $S_m$ is the set of functions constant on the partition $([(k-1)/d_m,k/d_m))_{k=1,...,d_m}$, with $d_m=50$. $n=100$, $N=100$ and $((X_i^J)_{i=1,...,n})_{J=1,...,N}$ are $N$ independent samples with common law $\p_s$. For all $J=1,...,N$, we compute the projection estimator $\hat{s}_m^J$ on $S_m$ with the sample $(X_i^J)_{i=1,...,n}$. Then, we take $n_b=10000$ resampling schemes $(W_1,...,W_n)$. For all resampling schemes, we compute the quantity
\begin{equation*}
p^J_{W}(S_m)=\frac1{v_W^2}\left(\sum_{\lambda\in \Lambda}[(P^{J,W}_n-\bar{W}_nP_n^J)\psi_{\lambda}]^2\right)
\end{equation*}
and we obtain an approximation of the $(1-\alpha)$-quantiles $\hat{q}^J_{\alpha}$ of its conditional distribution $\mathcal{D}^W(p^J_{W}(S_m))$. We plot the frequency of $J$ such that $\left\|s_m-\hat{s}^J_m\right\|^2\leq \hat{q}^J_{\alpha}$ and the function $f(\alpha)=\alpha$ when $\alpha$ varies in $(0.5,1)$ in the following curves.
\newpage
\begin{figure}[hbt]
\centering
\includegraphics[width=8cm,height=7cm]{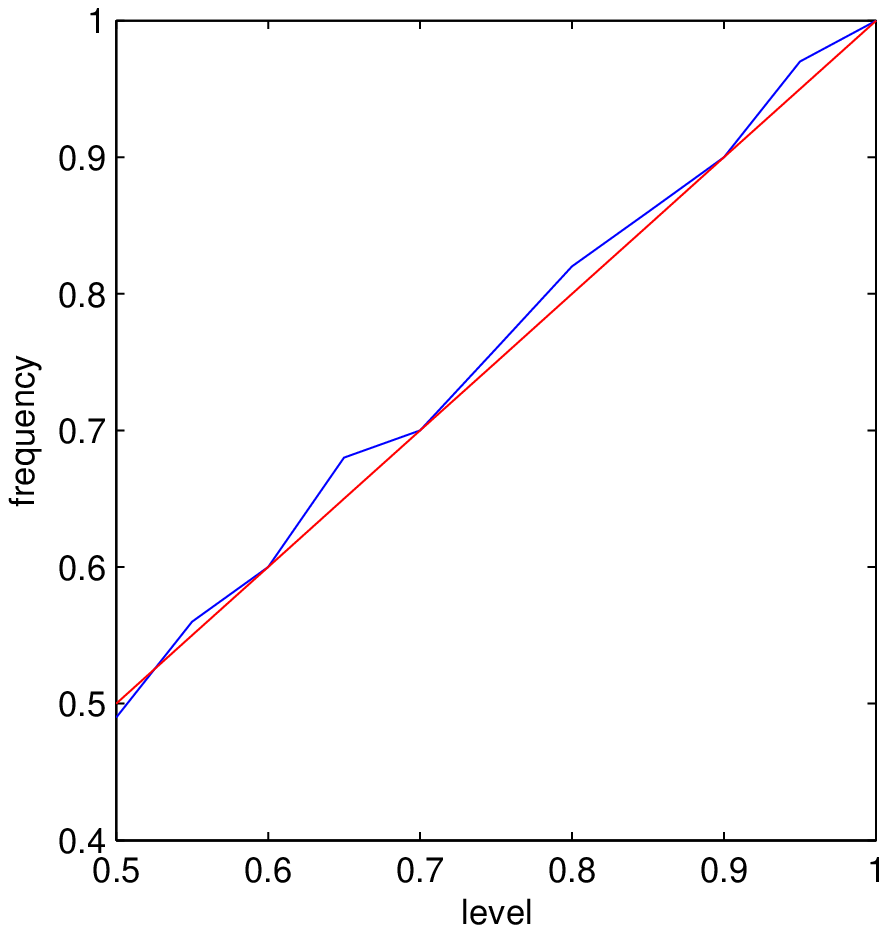}
\end{figure}

\noindent
{\bf{Comments}}
\begin{itemize}
\item The covering property of this empirical ball is very close to the one we would like to obtain. Hence, this method seems to give sharp confidence balls for $s_m$. The computation time is the same as in the first method.
\item We do not prove any theoretical evidence of this covering property. In particular, we cannot guarantee that $\p_s(\left\|s_m-\hat{s}_m\right\|_2^2\leq \hat{q}_{\alpha})\geq 1-\alpha$ occurs for any $n$. 
\end{itemize}
{Acknowledgements: } The author would like to thank gratefully B\'eatrice Laurent and Cl\'ementine Prieur for many fruitful advices.\\
He also would like to thank the reviewers and the associated editors who helped to improve a first version of the article.
\section{Proofs.}\label{Ch5S5}
\subsection{Proof of Theorem \ref{theo:CpW}}
The theorem can easily be deduced from the following Lemmas, whose proofs are postponed to the appendix.
\begin{lemma}\label{lem:ps-pW=U}
Let $X_{1},...,X_{n}$ be an i.i.d sample with common density $s$ in $L^{2}(\mu)$ and let $(\psi_{\lambda})_{\lambda\in\Lambda}$ be an orthonormal system in $L^{2}(\mu)$. Let $W_{1},...W_{n}$ be a resampling scheme, let $\bar{W_{n}}=n^{-1}\sum_{i=1}^{n}W_{i}$ and let $C_{W}=\Var(W_{1}-\bar{W_{n}})^{-1}$.\\
Let $T_{s}(\Lambda)=\sum_{\lambda\in\Lambda}(\psi_{\lambda}-P_{s}\psi_{\lambda})^{2}$,
$$p_{s}(\Lambda)=\sum_{\lambda\in\Lambda}\left[(P_{n}-P_{s})\psi_{\lambda}\right]^{2},\;p_{W}(\Lambda)=C_{W}\E_{W}\left(\sum_{\lambda\in\Lambda}\left[(P^{W}_{n}-\bar{W_{n}}P_{n})\psi_{\lambda}\right]^{2}\right),$$
$$U_{s}(\Lambda)=\frac1{n(n-1)}\sum_{i\neq j=1}^{n}\sum_{\lambda\in\Lambda}(\psi_{\lambda}(X_{i})-P_{s}\psi_{\lambda})(\psi_{\lambda}(X_{j})-P_{s}\psi_{\lambda}).$$
Then 
$$p_{s}(\Lambda)=\frac{1}nP_{n}T_{s}(\Lambda)+\frac{n-1}nU_{s}(\Lambda),\;p_{W}(\Lambda)=\frac{1}nP_{n}T_{s}(\Lambda)-\frac{1}nU_{s}(\Lambda),\;p_{s}(\Lambda)-p_{W}(\Lambda)=U_{s}(\Lambda).$$
\end{lemma} 

\begin{lemma}\label{lem:concU}
Let $X_{1},...,X_{n}$ be an i.i.d sample with common density $s$ in $L^{2}(\mu)$ and let $(\psi_{\lambda})_{\lambda\in\Lambda}$ be an orthonormal system in $L^{2}(\mu)$. Let $ D_{s,\Lambda}=\sum_{\lambda\in\Lambda}P_s\left((\psi_{\lambda}-P_s\psi_{\lambda})^2\right),$
$$U_{s}(\Lambda)=\frac1{n(n-1)}\sum_{i\neq j=1}^{n}\sum_{\lambda\in\Lambda}(\psi_{\lambda}(X_{i})-P_{s}\psi_{\lambda})(\psi_{\lambda}(X_{j})-P_{s}\psi_{\lambda}),$$
$$B(\Lambda)=\left\{\sum_{\lambda\in\Lambda}a_{\lambda}\psi_{\lambda};\;\sum_{\lambda\in\Lambda}a_{\lambda}^2\leq 1\right\},\;v_{s,\Lambda}^2=\sup_{t\in B(\Lambda)}P_s\left((t-Pt)^2\right),\;b_\Lambda=\sup_{t\in B(\Lambda)}\left\|t\right\|_{\infty}.$$
For all $\xi$ in $\{-1,1\}$, for all $x>0$, we have
$$\p_s\left(\xi U_{s}(\Lambda)>5.7v_{s,\Lambda}\frac{\sqrt{D_{s,\Lambda}x}}n+8v_{s,\Lambda}^2\frac xn+384\sqrt{2}v_{s,\Lambda}b_\Lambda\left(\frac xn\right)^{3/2}+2040b_\Lambda^2\left(\frac xn\right)^2\right)\leq ee^{-x}.$$
\end{lemma}

\begin{lemma}\label{lem:vbD}
Let $S$ be a linear space with finite dimension $d$ satisfying assumption {\bf H2}. Let $s$ be a density in $L^2(\mu)\cap L^{\infty}(\mu)$, let $(\psi_{\lambda})_{\lambda\in\Lambda}$ be an orthonormal basis of $S$. Let 
$$B(\Lambda)=\left\{\sum_{\lambda\in\Lambda}a_{\lambda}\psi_{\lambda};\;\sum_{\lambda\in\Lambda}a_{\lambda}^2\leq 1\right\},v_{s,\Lambda}^2=\sup_{t\in B(\Lambda)}P_s\left((t-Pt)^2\right),\;b_\Lambda=\sup_{t\in B(\Lambda)}\left\|t\right\|_{\infty},$$
$$ D_{s,\Lambda}=\sum_{\lambda\in\Lambda}P_s\left((\psi_{\lambda}-P_{s}\psi_{\lambda})^2\right)=P_s\left(\sup_{t\in B(\Lambda)}(t-P_{s}t)^2\right).$$
We have
$$v_{s,\Lambda}^2\leq \left\|s\right\|_\infty\wedge C_1\left\|s\right\|\sqrt{d},\;v_{s,\Lambda}^2\leq D_{s,\Lambda}\leq b^2_\Lambda\leq C_1^2d.$$
\end{lemma}

\noindent
Let us now explain briefly the proof of Theorem \ref{theo:CpW}. Let $X_{1},...,X_{n}$ be an i.i.d sample with common density $s$ in $L^2(\mu)\cap L^{\infty}(\mu)$. Let $(\psi_{\lambda})_{\lambda\in\Lambda_m}$ be an orthonormal basis in $S_m$. It comes from Lemmas \ref{lem:ps-pW=U} and \ref{lem:concU} that, using the notations of these lemmas, for all $x>0$, there exists an absolute constant $\kappa=2040$ such that, with probability larger than $1-e^{-x+1}$
\begin{equation}\label{eq:proof1}
\|s_m-\hat{s}_m\|^2\leq p_W(S_m)+\kappa\left(v_{s,\Lambda_m}\frac{\sqrt{D_{s,\Lambda_m}x}}{4n}+v_{s,\Lambda_m}^2\frac x{4n}+v_{s,\Lambda_m}b_{\Lambda_m}\left(\frac xn\right)^{3/2}+b_{\Lambda_m}^2\left(\frac xn\right)^2\right).
\end{equation}
Since $x\geq 2$, $\sqrt{x}\leq x$ and $x-1\geq x/2$. We have 
$$2v_{s,\Lambda_m}b_{\Lambda_m}\left(\frac xn\right)^{3/2}\leq v_{s,\Lambda_m}^2\frac xn+b_{\Lambda_m}^2\left(\frac xn\right)^2,\;v_{s,\Lambda_m}^2\leq D_{s,\Lambda_m}.$$
Hence, from (\ref{eq:proof1}), with probability larger than $1-e^{-x/2}$,
$$\|s_m-\hat{s}_m\|^2\leq p_W(S_m)+\kappa\left(v_{s,\Lambda_m}\frac{\sqrt{D_{s,\Lambda_m}}x}{n}+\frac32b_{\Lambda_m}^2\left(\frac xn\right)^2\right).$$
Since $\sqrt{d_m}x/n\leq C_3$, $d_mx^2/n^2\leq C_3\sqrt{d_m}x/n$, from Lemma \ref{lem:vbD}, 
\begin{equation}\label{eq:proof2}
v_{s,\Lambda_m}\frac{\sqrt{D_{s,\Lambda_m}}x}{n}+\frac32b_{\Lambda_m}^2\left(\frac xn\right)^2\leq C_1\left(\sqrt{ \left\|s\right\|_\infty\wedge C_1\left\|s\right\|\sqrt{d}\wedge C_1^2d}+\frac32C_1C_3\right)\frac{\sqrt{d_m}x}n.
\end{equation}
This concludes the proof of Theorem \ref{theo:CpW}, with $\kappa_v=2040C_1(1\vee C_1\vee 3C_1C_3/2)$.

\subsection{Proof of Corollary \ref{cor:Vm}}
We use a union bound to obtain that
\begin{eqnarray*}
&&\p_s\left(\exists m\in \M_n,\;\|s_m-\hat{s}_m\|^2> V(m,\beta,X_1,...,X_n)\right)\\
&&\leq N_n\max_{m\in\M_n}\p_s\left(\|s_m-\hat{s}_m\|^2> V(m,\beta,X_1,...,X_n)\right).
\end{eqnarray*}
All the models satisfy {\bf H2}. From assumption {\bf H3$(\M,\beta)$}, $x_n$ satisfies $2\leq x_n\leq C_3n/\sqrt{d_m}$ with $C_3=C_\M$, thus, from Theorem \ref{theo:CpW}, for all $m$ in $\M_n$,
$$\p_s\left(\|s_m-\hat{s}_m\|^2> V(m,\beta,X_1,...,X_n)\right)\leq e^{-x_n/2}.$$
Finally, ${\rm Card}(\M_n)e^{-x_n/2}\leq \frac{\beta}2$, which concludes the proof of Corollary \ref{cor:Vm}.

\subsection{Proof of Theorem \ref{theo:SizeCB}}
Let $s$ be a density in $L^2(\mu)\cap L^{\infty}(\mu)$, we only have to prove that there exists a constant $\kappa$ such that, with $\p_s$-probability larger than $1-\alpha$,
$$\forall m\in \M_n,\;p_W(S_m)\leq \kappa\left(\frac{d_m}n+\left(1+\sqrt{\left\|s\right\|_{\infty}\wedge \|s\|d_m^{1/2}\wedge d_m}\right)\frac{\sqrt{d_m}}n\ln\left[\frac{N_n}{\alpha}\right]\right).$$
Let $(\psi_{\lambda})_{\lambda\in\Lambda_m}$ be an orthonormal basis of $S_m$, from Lemma \ref{lem:ps-pW=U} and using the notations of this lemma, 
$$p_{W}(\Lambda)=\frac{1}nP_{n}T_{s}(\Lambda_m)-\frac{1}nU_{s}(\Lambda_m).$$
We follow the proof of Theorem \ref{theo:CpW}. From Lemmas \ref{lem:concU} and \ref{lem:vbD} and assumptions {\bf H1, H2, H3$(\M,\alpha)$}, there exists a constant $\kappa$ such that 
$$\p_s\left(\exists m\in \M_n,\;U_{s}(\Lambda_m)> \kappa\sqrt{\left\|s\right\|_{\infty}\wedge \|s\|d_m^{1/2}\wedge d_m}\frac{\sqrt{d_m}\ln[ N_n/\alpha]}n\right)\leq \alpha.$$
Moreover, it is easy to check, with Cauchy-Schwarz inequality, that, using the notations of Lemma \ref{lem:vbD}
$$T_{s}(\Lambda_m)=\sup_{t\in B(\Lambda_m)}(t-P_st)^2.$$
Hence, using assumptions {\bf H2}, we obtain
$$P_{n}T_{s}(\Lambda_m)\leq \left\|T_{s}(\Lambda_m)\right\|_{\infty}\leq 2C_1^2d_m.$$
This conclude the proof of Theorem \ref{theo:SizeCB}.

\subsection{Proof of Lemma \ref{lem:concbiais}}
Let $X_{1},...,X_{n}$ be an i.i.d sample with common density $s$ in $L^2(\mu)\cap L^{\infty}(\mu)$. Let  $(\psi_{\lambda})_{\lambda\in\Lambda_n}$ be an orthonormal basis of $S_n$ such that  $(\psi_{\lambda})_{\lambda\in\Lambda_m}$ is an orthonormal basis of $S_m$, with $\Lambda_m\subset \Lambda_n$. The Hoeffding's decomposition of the $U$-statistic $p_b(S_m,S_n)$ can be written
\begin{eqnarray*}
p_b(S_m,S_n)&=&U_s(\Lambda_n-\Lambda_m)+2P_n\left(\sum_{\lambda\in\Lambda_n-\Lambda_m}(P_s\psi_{\lambda})(\psi_{\lambda}-P_s\psi_{\lambda}) \right)+\sum_{\lambda\in\Lambda_n-\Lambda_m}(P_s\psi_{\lambda})^2\\
&=&U_s(\Lambda_n-\Lambda_m)+2(P_n-P_s)\left(s_n-s_m \right)+\|s_n-s_m\|^2,
\end{eqnarray*}
where, as usually, for all indexes sets $\Lambda$,
$$U_{s}(\Lambda)=\frac1{n(n-1)}\sum_{i\neq j=1}^{n}\sum_{\lambda\in\Lambda}(\psi_{\lambda}(X_{i})-P_{s}\psi_{\lambda})(\psi_{\lambda}(X_{j})-P_{s}\psi_{\lambda}).$$
It comes from Lemmas \ref{lem:concU} and \ref{lem:vbD} that, for all $2\leq x\leq C_3n/\sqrt{d_n}$,
$$\p_s\left(|U_s(\Lambda_n-\Lambda_m)|>\kappa_v(C_1,C_3)\left(1+\sqrt{\left\|s\right\|_{\infty}\wedge \left\|s\right\|d_n^{1/2}}\right)\frac{\sqrt{d_n}x}{n} \right)\leq 2e^{-x/2}.$$
If $s_n=s_m$, this concludes the proof. Else, let $\epsilon$ in $(0,1)$, the inequality $2ab\leq \epsilon a^2+\epsilon^{-1}b^2$ gives
$$2|(P_n-P_s)\left(s_n-s_m \right)|\leq \epsilon\|s_n-s_m\|^2+\epsilon^{-1}\left((P_n-P_s)\left(\frac{s_n-s_m}{\|s_n-s_m\|} \right)\right)^2.$$
The function $s_{m,n}=(s_n-s_m)/\|s_n-s_m\|$ satisfies $\|s_{m,n}\|\leq 1$ and, from Bernstein's inequality, for all $x>0$,
$$\p_s\left(|(P_n-P_s)\left(s_{n,m} \right)|>\sqrt{2P_s\left[(s_{m,n}-P_s s_{m,n})^2\right]\frac xn}+\left\|s_{n,m}\right\|_\infty\frac x{3n}\right)\leq 2e^{-x}.$$
Since $s_{m,n}$ belongs to $S_n$, which satisfies {\bf H2}, it comes from Lemma \ref{lem:vbD} that
$$P_s\left[(s_{m,n}-P_s s_{m,n})^2\right]\leq \left(\left\|s\right\|_{\infty}\wedge C_1\left\|s\right\|d_n^{1/2}\right),\;\left\|s_{n,m}\right\|_\infty\leq C_1\sqrt{d_n}.$$
We conclude the proof of Lemma  \ref{lem:concbiais} saying that $x\geq 2$ implies $2e^{-x}\leq e^{-x/2}$. In this Lemma, we proved that we can choose $\kappa_b(\epsilon,C_3)=\kappa_v(C_1,C_3)+2\epsilon^{-1}(2\vee 2C_1\vee C_3C_1^2/9).$

\subsection{Proof of Corollary \ref{cor:estbiais}}
Let $X_1,...,X_n$ be an iid sample with common density $s$ in $B_{2,\infty}(M_2,M_\infty,0,L^2(\mu))$. Let $\epsilon$ in $(0,1)$ and let $\Omega_n(\epsilon)$ denote the event
\begin{align*}
\left\{\vphantom{\frac{\sqrt{d_n}x_n}{n}}\forall m\in \M_n,\;\left|p_b(S_m,S_n)-\|s_n-s_m\|^2\right|\leq\right.& \left.\epsilon\|s_n-s_m\|^2\right.\\
&\left.+\kappa_b(\epsilon,C_\M)\sqrt{\left\|s\right\|_{\infty}\wedge \left\|s\right\|d_n^{1/2}}\frac{\sqrt{d_n}x_n}{n}\right\}.
\end{align*}
A union bound gives that $\p_s(\Omega_n(\epsilon)^c)$ is upper bounded by the sum over $\M_n$ of
\begin{equation*}
\p_s\left(\left|p_b(S_m,S_n)-\|s_n-s_m\|^2\right|> \epsilon\|s_n-s_m\|^2+\kappa_b(\epsilon,C_\M)\sqrt{\left\|s\right\|_{\infty}\wedge \left\|s\right\|d_n^{1/2}}\frac{\sqrt{d_n}x_n}{n}\right).
\end{equation*}
Assumption {\bf H3$(\M,\beta)$} ensures that $x_n$ satisfies $2\leq x_n\leq C_3n/\sqrt{d_m}$ with $C_3=C_\M$, thus, Lemma \ref{lem:concbiais} gives that this last probability is upper bounded by $3e^{-x_n/2}$. Our choice of $x_n$ ensures that $3N_ne^{-x_n/2}\leq\beta/2$ and thus that $\p_s(\Omega_n(\epsilon)^c)\leq \frac{\beta}2.$ The proof of Corollary \ref{cor:estbiais} is concluded because, on $\Omega_n(\epsilon)$,
$$(1-\epsilon)\|s_n-s_m\|^2\leq p_b(S_m,S_n)+\kappa_b(\epsilon,C_\M)\sqrt{\left\|s\right\|_{\infty}\wedge \left\|s\right\|d_n^{1/2}}\frac{\sqrt{d_n}x_n}{n}.$$
\subsection{Proof of Theorem \ref{theo.ACS}}
The theorem is a straightforward consequence of Corollaries \ref{cor:Vm} and \ref{cor:estbiais}.
\subsection{Proof of Theorem \ref{theo.MinMaxlb}}
We begin the proof with the following proposition, which shows that $\phi_n(\alpha,\beta,S_m,S_m)\geq d_m/(12n)$. Since $\phi_n(\alpha,\beta,S_n,S_m)\geq\phi_n(\alpha,\beta,S_m,S_m)$, the same bound holds also for $\phi_n(\alpha,\beta,S_n,S_m)$.
\begin{prop}\label{prop:optivar}
Let $S$ be the set of histograms on the partition, 
$$\left\{\left[\frac kd,\frac{k+1}d\right),\;k=0,...,d-1\right\}.$$ 
Let $X_1,...,X_n$ be an i.i.d sample. Let $\alpha,\beta$ be real numbers in $(0,1)$ such that $\alpha+\beta<1$. Assume that $d\geq 3+18\log(\sqrt{2}/(1-\alpha-\beta))$, then
\begin{equation*}
\phi_n(\alpha,\beta,S,S)\geq \frac{d}{12n}.
\end{equation*}
\end{prop}
The proof is decomposed in two lemmas.
\begin{lemma}\label{7}
Let $\hat{B}_{\beta}=B_2(\hat{s},\hat{\rho}_{\beta},S)$ in $CS(S,\beta)$ and let $\rho_{\alpha,\beta}$ be a real number such that
$$\forall s\in S,\; \p_s\left(\hat{\rho}_{\beta}\leq \rho_{\alpha,\beta}\right)\geq 1-\alpha.$$
Then,
\begin{equation}\label{partiel}
\forall s\in S,\;\p_s\left(\left\|s-\hat{s}\right\|>\rho_{\alpha,\beta}\right)\leq \alpha +\beta.
\end{equation}
\end{lemma}
\begin{proof}{of Lemma \ref{7}:}
\begin{eqnarray*}
\p_s\left[\left\|s-\hat{s}\right\|>\rho_{\alpha,\beta}\right]&=&\p_s\left[\left\|s-\hat{s}\right\|>\rho_{\alpha,\beta}\cap \rho_{\alpha,\beta}\geq\hat{\rho}_{\beta}\right]\\
&&+\p_s\left[\left\|s-\hat{s}\right\|>\rho_{\alpha,\beta}\cap \rho_{\alpha,\beta}< \hat{\rho}_{\beta}\right]\\
&\leq&\p_s\left[\left\|s-\hat{s}\right\|>\hat{\rho}_{\beta}\right]+\p_s\left[\rho_{\alpha,\beta} < \hat{\rho}_{\beta}\right]\leq\alpha+\beta.
\end{eqnarray*}
\end{proof}
\begin{lemma}\label{8}
Let $\delta=\alpha+\beta$ and let $\rho_\delta$ be any real number satisfying (\ref{partiel}). Then we have
$$\rho_\delta^2\geq \frac{d-1}{2n}-\frac1n\sqrt{2(d+1)\ln\left[\frac{\sqrt{1+(d+1)n^{-1}}}{1-\delta}\right]}.$$
\end{lemma}
{\bf{Remark:}} When $d\geq 3+18\log(\sqrt{2}/(1-\delta))$ and $n\geq d+1$, we have $$\sqrt{2(d+1)\ln\left[\frac{\sqrt{1+(d+1)n^{-1}}}{1-\delta}\right]}\leq \frac{d-1}{3},$$
thus $\rho_\delta^2\geq (d-1)/(6n)\geq d/(12n)$.

\vspace{0.2cm}

\noindent
{\bf Proof:} We prove that if $$\rho_\delta^2= \frac{d-1}{2n}-\frac1n\sqrt{2(d+1)\ln\left[\frac{\sqrt{1+(d+1)n^{-1}}}{1-\delta}\right]}$$ 
then 
\begin{equation*}
\inf_{s\in S}\p_s\left[\left\|s-\hat{s}\right\|\leq \rho_\delta\right]\leq 1-\delta.
\end{equation*}
Let $s_0 =1_{[0,1)}$, $\Lambda=\{1,...,[d/2]\}$ and for all $\lambda$ in $\Lambda$, let 
$$\psi_{\lambda}=\sqrt{\frac{d}2}\left(1_{[2(\lambda-1)/d,(2\lambda-1)/d)}-1_{[(2\lambda-1)/d,2\lambda/d)}\right).$$ 
It is easy to check that $(\psi_{\lambda})_{\lambda\in \Lambda}$ is an orthonormal system in $S$, orthogonal to $s_0$ such that, for all $\lambda$ in $\Lambda$, $\left\|\psi_{\lambda}\right\|_{\infty}\leq \sqrt{d/2}$. Let $\hat{s}_0=\int\hat{s}s_0d\mu$ and for all $\lambda$ in $\Lambda$, let 
$$\hat{s}_{\lambda}=\int\hat{s}\psi_{\lambda}d\mu.$$ 
Let $(\xi_{\lambda})_{\lambda\in \Lambda}$ be independent Rademacher random variables, independent of $X_1,...,X_n$, let $\rho$ be some real number to be chosen later and let $s_{\xi}=s_0+\rho\sum_{\lambda\in \Lambda}\xi_{\lambda}\psi_{\lambda}.$ The $\psi_{\lambda}$ have distinct support, thus $\left\|\sum_{\lambda\in \Lambda}|\psi_{\lambda}|\right\|_{\infty}\leq \sqrt{d/2}$ and $s_{\xi}$ is a density if 
\begin{equation}\label{Cond:1}
-\sqrt{\frac 2d}\leq \rho \leq \sqrt{\frac 2d}
\end{equation} 
Assume that (\ref{Cond:1}) holds, then
\begin{eqnarray}
\inf_{s\in S}\p_s\left[\left\|s-\hat{s}\right\|\leq \rho_\delta\right]&\leq&\p_{s_{\xi}}\left[\left\|s_{\xi}-\hat{s}\right\|\leq \rho_\delta\right].\label{de1}
\end{eqnarray}
We have 
\begin{eqnarray}
\left\|s_{\xi}-\hat{s}\right\|^2&=&(1+s_0)^2+\sum_{\lambda\in \Lambda}\left(\rho\xi_{\lambda}-\hat{s}_{\lambda}\right)^2\nonumber\\
&=&\sum_{\lambda\in \Lambda,\;\rho\xi_{\lambda}\hat{s}_{\lambda}\leq 0}\rho^2-2\rho\xi_{\lambda}\hat{s}_{\lambda}+\hat{s}_{\lambda}^2\geq\rho^2N(\xi,\hat{s}),\label{de2}
\end{eqnarray}
where $N(\xi,\hat{s})=\textrm{Card}(\{\lambda\in \Lambda,\;\rho\xi_{\lambda}\hat{s}_{\lambda}\leq 0\})=\sum_{\lambda\in \Lambda}1_{\{\rho\xi_{\lambda}\hat{s}_{\lambda}\leq 0\}}$. If we plug (\ref{de2}) in (\ref{de1}), we obtain
$$\inf_{s\in S}\p_s\left[\left\|s-\hat{s}\right\|_2\leq \rho_\delta\right]\leq\int_0^1 {\bf 1}_{\rho^2N(\xi,\hat{s})\leq \rho_\delta}s_{\xi}d\mu.$$
We integrate with respect to $\xi$ and we apply Fubini's theorem to obtain 
\begin{equation}\label{de3}
\inf_{s\in S}\p_s\left[\left\|s-\hat{s}\right\|_2\leq\rho_\delta^2\right]\leq\p_{s_{\xi}}\left[ \rho^2N(\xi,\hat{s})\leq\rho_\delta^2\right]= \leq\int_0^1\E_{\xi}\left({\bf 1}_{\rho^2N(\xi,\hat{s})\leq \rho_\delta^2}s_{\xi}\right)d\mu.
\end{equation}
From Cauchy-Schwarz inequality,
\begin{equation}\label{CS}
\E^2_{\xi}\left({\bf 1}_{\rho^2N(\xi,\hat{s})\leq \rho_\delta^2}s_{\xi}\right)\leq \p_{\xi}\left(\rho^2N(\xi,\hat{s})\leq \rho_\delta^2\right)\E_{\xi}\left(s^2_{\xi}\right),
\end{equation}
and $\E_{\xi}s^2_{\xi}=s_0^2+\rho^2\sum_{\lambda\in \Lambda}\psi_{\lambda}^2$. For all $\lambda$ in $\Lambda$, $\int_0^1\psi_{\lambda}^2= 1$, thus
\begin{equation}\label{de4}
\int_0^1\E_{\xi}s^2_{\xi}d\mu=1+\rho^2\left[\frac d2\right].
\end{equation}
Moreover, conditionally to $\hat{s}$, $N(\xi,\hat{s})$ is a sum of $[d/2]$ independent random variables valued in $\left\{0,1\right\}$. Thus, from Hoeffding's inequality,
\begin{equation}\label{Hoef}
\forall t>0,\;\p_{\xi}\left(N(\xi,\hat{s})\leq\E_{\xi}\left(N(\xi,\hat{s})\right)-\sqrt{\left[\frac d2\right]t}\right)\leq e^{-2t}.
\end{equation}
In (\ref{Hoef}), we have $E_{\xi}\left(N(\xi,\hat{s})\right)=\sum_{\lambda\in \Lambda}\E_{\xi}\left({\bf 1}_{\xi_{\lambda}\hat{s}_{\lambda}\leq 0}\right)\geq[d/2]/2$ and we choose
$$t=\ln\left[ \frac{\sqrt{1+\rho^2[d/2]}}{1-\delta}\right],\;\rho=\sqrt{\frac{2}{n}}\leq\sqrt{ \frac2{d}}.$$
Since $(d-1)/2\leq [d/2]\leq (d+1)/2$,
$$t\leq \ln\left[ \frac{\sqrt{1+(d+1)/n}}{1-\delta}\right],\;E_{\xi}\left(N(\xi,\hat{s})\right)\geq\frac{d-1}{4}.$$
Thus
$$\{\rho^2N(\xi,\hat{s})\leq \rho_\delta^2\}\subset\{N(\xi,\hat{s})\leq\E_{\xi}\left(N(\xi,\hat{s})\right)-\sqrt{[d/2]t}\}.$$  
Hence, from (\ref{Hoef}),
\begin{equation}\label{de5}
\p_{\xi}\left(\rho^2N(\xi,\hat{s})\leq \rho_\delta^2\right)\leq \frac{(1-\delta)^2}{1+\rho^2[d/2]}.
\end{equation}
We plug inequalities (\ref{de4}) and (\ref{de5}) in (\ref{CS}) to obtain 
$$\int_0^1\E^2_{\xi}\left({\bf 1}_{d\rho^2N(\xi,\hat{s})\leq \rho_\delta^2}s_{\xi}\right)\leq (1-\delta)^2.$$ 
Thus, from (\ref{de3}) and Jensen inequality,
\begin{equation*}
\inf_{s\in S}\p_s\left[\left\|s-\hat{s}\right\|_2\leq\rho_\delta\right]\leq 1-\delta.
\end{equation*}

We already know thanks to Proposition \ref{prop:optivar} that $\phi_n(\alpha,\beta,S_n,S_m)\geq d_m/(12n)$, therefore, it remains to prove that  $\phi_n(\alpha,\beta,S_n,S_m)\geq \sqrt{d_n}/n$. Let $s_0=I_{[0,1]}$, let $\hat{B}_\beta=B_2(\hat{s},\hat{\rho}_{\beta},S_n)$ be a confidence ball in $CS(S_n,\beta)$ and let $\rho_{\alpha,\beta}>0$ such that for all densities $s$ in $S_m$,
$$\p_s\left(\hat{\rho}_{\beta}\leq \rho_{\alpha,\beta}\right)\geq 1-\alpha.$$
We will prove that $\rho_{\alpha,\beta}\geq c\sqrt{d_n}/n$, which is sufficient to prove Theorem \ref{theo.MinMaxlb}. We decompose the proof into two lemmas. 
\begin{lemma}\label{4}
Let $S_n(\rho_{\alpha,\beta})=\left\{t\in S_n \; ; \; \left\|t-s_0\right\|_2\geq 2\rho_{\alpha,\beta}\right\}$. There exists a test $T$ of null hypothesis $H_0: s=s_0$ against the alternative $H_1: s\in S_n(\rho_{\alpha,\beta})$ with confidence level more than $1-\beta$ and power more than $1-\alpha-\beta$, ie such that 
\begin{equation*}
\p_{s_0}(T=0)\geq 1-\beta,\;\inf_{s\in S_n(\rho_{\alpha,\beta})}\p_s(T=1)\geq 1-(\alpha+\beta).
\end{equation*}
\end{lemma}
\begin{proof}{of Lemma \ref{4}:}
Let $T=1_{s_0\in \hat{B}_\beta}$. Since $s_0$ belongs to $S_n$ and $\hat{B}_\beta$ belongs to $CS(S_n,\beta)$,  $\p_{s_0}(T=0)\geq 1-\beta$. Moreover, for all $s$ in $S_n(\rho_{\alpha,\beta})$,
\begin{eqnarray*}
\p_s(T=0)&=&\p_s(s_0\in \hat{B}_\beta)=\p_s(\left\|s_0-\hat{s}\right\|\leq \hat{\rho}_{\beta})\\
&\leq&\p_s(\left\|s_0-s\right\|-\left\|s-\hat{s}\right\|\leq \hat{\rho}_{\beta})\leq \p_s(\left\|s-\hat{s}\right\|\geq 2\rho_{\alpha,\beta}-\hat{\rho}_{\beta}).
\end{eqnarray*}
This last probability is equal to
\begin{eqnarray*}
&&\p_s(\left\|s-\hat{s}\right\|\geq 2\rho_{\alpha,\beta}-\hat{\rho}_{\beta}\cap \hat{\rho}_{\beta}>\rho_{\alpha,\beta})+\p_s(\left\|s-\hat{s}\right\|\geq  2\rho_{\alpha,\beta}-\hat{\rho}_{\beta}\cap \hat{\rho}_{\beta}\leq \rho_{\alpha,\beta})\\
&&\leq\p_s(\hat{\rho}_{\beta}>\rho_{\alpha,\beta})+\p_s(\left\|s-\hat{s}\right\|\geq \hat{\rho}_{\beta})\leq \beta+\alpha. \square
\end{eqnarray*}
\end{proof}
The second lemma gives the separation rate for the test of null hypothesis $H_0:s=s_0$
\begin{lemma}\label{5}
Let $\eta=2(1-2\alpha-\beta)$, let $\rho>0$. Let $\Theta_{\alpha}$ be the set of tests $T_{\alpha}$ with confidence level $\alpha$, of null hypothesis $H_0: s=s_0$ against the alternative $H_1:s\in S_n(\rho)$, where $S_n(\rho)$ is the set of all densities $s$ in $S_n$ such that $\left\|s-s_0\right\|\geq \rho$.\\
Let $\beta\left(S_n(\rho)\right)=\inf_{T_{\alpha}\in\Theta_{\alpha}}\sup_{s\in S_n(\rho)}\p_s(T_{\alpha}=0)$.\\
If $d_n\geq 10$ and $\rho^2<\sqrt{\ln(1+\eta^2)/3.2}(\sqrt{d_n-1}/n)$ then $\beta\left(S(\rho)\right)>\beta+\alpha$.
\end{lemma}
{\bf{Comments:}} From Lemmas \ref{4} and \ref{5}, we deduce that 
$$\rho_{\alpha,\beta}^2\geq \sqrt{\frac{\ln(1+\eta^2)}{3.2}}\frac{\sqrt{d_n-1}}{4n}\geq \frac{\sqrt{\ln(1+\eta^2)}}{11}\frac{\sqrt{d_n}}{n}.$$ 
Thus the proof of Lemma \ref{5} concludes the proof of Theorem \ref{theo.MinMaxlb}.
\begin{proof}{of lemma \ref{5}:}
The function $\beta\left(S_n(\rho)\right)$ is non-increasing with $\rho$. Thus we take 
$$\rho^2=\sqrt{\ln(1+\eta^2)/3.2}\sqrt{d_n-1}/n$$
and we will to prove that $\beta\left(S_n(\rho)\right)\geq \alpha+\beta.$
Let $\mu_\rho$ be a probability measure on $S_n(\rho)$, let $P_{\mu_\rho}=\int P_sd\mu_\rho$. 
\begin{eqnarray}
\beta\left(S_n(\rho)\right)&\geq&\inf_{T_{\alpha}\in\Theta_{\alpha}}\p_{\mu_\rho}(T_{\alpha}=0) \nonumber\\
&=&\inf_{T_{\alpha}\in\Theta_{\alpha}}\left(\p_{\mu_\rho}(T_{\alpha}=0)-\p_{s_0}(T_{\alpha}=0)+\p_{s_0}(T_{\alpha}=0)\right)\nonumber\\
&\geq&1-\alpha+\inf_{T_{\alpha}\in\Theta_{\alpha}}\left(\p_{\mu_\rho}(T_{\alpha}=0)-\p_{s_0}(T_{\alpha}=0)\right)\\
&\geq&1-\alpha-\sup_{A\; ; \; \p_{s_0}(A)\leq \alpha}\left|\p_{\mu_\rho}(A)-\p_{s_0}(A)\right| \nonumber\\
&\geq&1-\alpha-1/2\left\|\p_{\mu_\rho}-\p_{s_0}\right\|_{TV} \label{min}
\end{eqnarray}
where $\left\|.\right\|_{TV}$ denote the total variation distance. Assume that $\p_{\mu_\rho}$ is absolutely continuous with respect to $\p_{s_0}$. Let $L_{\mu_\rho}=d\p_{\mu_\rho}/d\p_{s_0}$, then 
$$\left\|\p_{\mu_\rho}-\p_{s_0}\right\|_{TV}=\E_{s_0}\left|L_{\mu_\rho}(X_1,...,X_n)-1\right|\leq\left(\p_{s_0}\left(L_{\mu_\rho}^2\right)-1\right)^{1/2}$$
and then
\begin{equation}\label{Maj}
\beta\left(S_n(\rho)\right)\geq 1-\alpha-\frac{\sqrt{\E_{s_0}\left(L_{\mu_\rho}^2\right)-1}}2.
\end{equation}
From (\ref{Maj}), $\beta\left(S_n(\rho)\right)\geq \alpha+\beta$ if $\E_{s_0}\left(L_{\mu_\rho}^2\right)\leq 1+\eta^2$. Let us now give a probability measure on $S_n(\rho)$, absolutely continuous with respect to $P_{s_0}$, such that $\E_{s_0}\left(L_{\mu_\rho}^2\right)\leq 1+\eta^2$.\\
Let $(\psi_{\lambda})_{\lambda=1,...,[d_n/2]}$ be the following orthonormal system. Let $\psi_{0}=s_0$, $\phi=1_{[0,1/2)}-1_{[1/2,1)}$ and for all $\lambda=1,...,[d_n/2]$, $\psi_{\lambda}=\sqrt{d_n/2}\phi(d_nx/2-(\lambda-1))$. Let $\xi=(\xi_{\lambda})_{\lambda=1,...,[d_n/2]}$ be independent Rademacher random variables and let $\mu_\rho$ be the distribution of $s_{\xi}=s_0+\rho\sum_{\lambda=1}^{[d_n/2]}\xi_{\lambda}\psi_{\lambda}/\sqrt{[d_n/2]}$. Let us check that $\mu_\rho$ satisfies the required properties.
The functions  $(\psi_{\lambda})_{\lambda=1,...,[d_n/2]}$ have distinct support, thus 
$$\left\|\sum_{\lambda=1}^{[d_n/2]}|\psi_{\lambda}|\right\|_{\infty}\leq \sqrt{d_n/2}.$$ 
$s_{\xi}$ is a real density if $\rho\leq 1$. Since $2\alpha+\beta<1$, $\eta^2\leq 4$ and $\ln(1+\eta^2)\leq \ln(5)$. $\sqrt{d_n}\leq n$, hence
$$\rho^2\leq \sqrt{\frac{\ln(5)}{3.2}}\frac{\sqrt{d_n-1}}n\leq 1.$$
Since $(\psi_{\lambda})_{\lambda=1,..,[d_n/2]}$ is an orthonormal system, $\left\|s_{\xi}-s_0\right\|=\rho$, thus $s_{\xi}$ belongs to $S_n(\rho)$ and $\mu_\rho$ is a law on $S_n(\rho)$. Moreover 
$$\frac{d\p_{s_{\xi}}}{d\p_{s_0}}(x_1,..,x_n)=\prod_{\alpha=1}^n\left(1+\frac{\rho}{\sqrt{[d_n/2]}}\sum_{\lambda=1}^{[d_n/2]}\xi_{\lambda}\psi_{\lambda}(x_{\alpha})\right).$$
Thus
$$L_{\mu_{\rho}}(x_1,..,x_n)=\frac{1}{2^{[d_n/2]}}\sum_{\xi\in \{-1,1\}^{[d_n/2]}}\prod_{\alpha=1}^n\left(1+\frac{\rho}{\sqrt{[d_n/2]}}\sum_{\lambda=1}^{[d_n/2]}\xi_{\lambda}\psi_{\lambda}(x_{\alpha})\right).$$
Hereafter, in order to symplify the notations, we write $\sum_{\xi}$ instead of $\sum_{\xi\in \{-1,1\}^{[d_n/2]}}$ and $\sum_{\lambda}$ instead of $\sum_{\lambda=1}^{[d_n/2]}$. Let $\phi(\rho,\xi)=\rho\sum_{\lambda}\xi_{\lambda}\psi_{\lambda}/\sqrt{[d_n/2]}$, we have
\begin{eqnarray*}
L_{\mu_{\rho}}^2(x_1,..,x_n)&=&\frac{1}{2^{2([d_n/2])}}\sum_{\xi,\xi'}\prod_{\alpha=1}^n\left(1+\phi(\rho,\xi)(x_{\alpha})\right)\left(1+\phi(\rho,\xi')(x_{\alpha})\right).\\
\E_{s_0}(L_{\mu_{\rho}}^2)&=&\frac{1}{2^{2[d_n/2]}}\sum_{\xi}\sum_{\xi'}\prod_{\alpha=1}^nP_{s_0}\left(1+\phi(\rho,\xi)+\phi(\rho,\xi')+\phi(\rho,\xi)\phi(\rho,\xi')\right).
\end{eqnarray*}
For all $\lambda\neq\lambda'=1,...,[d_n/2]$, $\psi_{\lambda}\psi_{\lambda'}=0$, thus 
$$\phi(\rho,\xi)\phi(\rho,\xi')=\frac{\rho^2}{[d_n/2]}\left(\sum_{\lambda}\xi_{\lambda}\psi_{\lambda}\right)\left(\sum_{\lambda}\xi'_{\lambda}\psi_{\lambda}\right)=\frac{\rho^2}{[d_n/2]}\sum_{\lambda}\xi_{\lambda}\xi'_{\lambda}\psi^2_{\lambda}.$$
For all $\lambda=1,...,[d_n/2]$ and all $\alpha=1,...,n$, $P_{s_0}(\psi_{\lambda})=0$, $P_{s_0}(\psi^2_{\lambda})=1$, thus
\begin{eqnarray*}
\E_{s_0}(L_{\mu_{\rho}}^2)&\leq&\frac{1}{2^{2[d_n/2]}}\sum_{\xi}\sum_{\xi'}\left(1+\frac{\rho^2}{[d_n/2]}\sum_{\lambda}\xi_{\lambda}\xi'_{\lambda}\right)^n\\
&=&\frac{1}{2^{2[d_n/2]}}\sum_{\xi}\sum_{l=0}^{[d_n/2]}\sum_{\xi';\textrm{Card}(\lambda,\;\xi'_{\lambda}=\xi_{\lambda})=l}\left[1+\frac{\rho^2}{[d_n/2]}(2l-[d_n/2])\right]^n\\
&=&\frac{1}{2^{[d_n/2]}}\sum_{l=0}^{[d_n/2]}C_{[d_n/2]}^l\left[1+\frac{\rho^22l}{[d_n/2]}-\rho^2\right]^n
\end{eqnarray*}
For all real numbers $u\geq -1$, we have $0\leq 1+u\leq e^u$, thus $(1+u)^n\leq e^{nu}$. Since $\rho^2\leq1$, we can apply this inequality to all the $u_l=(2l/[d_n/2]-1)r^2$ and we obtain
\begin{equation*} 
\E_{s_0}(L_{\mu_{\rho}}^2)\leq\frac{1}{2^{[d_n/2]}}\sum_{l=0}^{[d_n/2]}C_{[d_n/2]}^l \exp\left(\frac{\rho^22nl}{[d_n/2]}-n\rho^2\right)=\frac{e^{-n\rho^2}}{2^{[d_n/2]}}\left(\exp\left(\frac{\rho^22n}{[d_n/2]}\right)+1\right)^{[d_n/2]}
\end{equation*}
Thus, $\E_{s_0}\left(L_{\mu_\rho}^2\right)\leq 1+\eta^2$ if
$$-n\rho^2+([d_n/2])\ln\left(\frac{\exp\left(\frac{\rho^22n}{[d_n/2]}\right)+1}{2}\right)\leq\ln(1+\eta^2).$$
For all positive $u$, $\ln(1+u)\leq u$, thus, we only have to prove that
$$-n\rho^2+\frac{[d_n/2]}{2}\left(\exp\left(\frac{\rho^22n}{[d_n/2]}\right)-1\right)\leq\ln(1+\eta^2).$$
$[d_n/2]\geq (d_n-1)/2$ and $d_n\geq 10$, thus
$$\frac{\rho^22n}{[d_n/2]}=2\sqrt{\frac{\ln(1+\eta^2)}{3.2}}\frac{\sqrt{d_n-1}}{[d_n/2]}\leq \frac{4*0.71}{\sqrt{d_n-1}}\leq 1.$$
For all real numbers $x$ in $[0,1]$, we have $e^x\leq 1+x+3.2x^2$, thus $\exp\left(\rho^22n/([d_n/2])\right)-1\leq \rho^22n/([d_n/2])+3.2\left(\rho^2n/([d_n/2])\right)^2$. Hence 
$$-n\rho^2+\frac{[d_n/2]}{2}\left(\exp\left(\frac{\rho^22n}{[d_n/2]}\right)-1\right)\leq 1.6\rho^4n^2/([d_n/2])\leq \frac{d_n-1}{2[d_n/2]}\ln(1+\eta^2)\leq\ln(1+\eta^2).$$
\end{proof}

\section{Appendix}
\subsection{Proof of Lemma \ref{lem:ps-pW=U}}
$\sum_{i=1}^n(W_i-\bar{W_n})=0$, thus, for all $\lambda$ in $\Lambda$, $(P^W_n-\bar{W_n}P_n)(P_s\psi_{\lambda})=0.$ Moreover, since the weights are exchangeable,
\begin{eqnarray*}
0&=&\E\left[\left(\sum_{i=1}^n(W_i-\bar{W_n})\right)^2\right]\\
&=&\sum_{i=1}^n\E\left((W_i-\bar{W_n})^2\right)+\sum_{i\neq j=1}^n\E(W_i-\bar{W_n})(W_j-\bar{W_n})\\
&=&n\E\left((W_1-\bar{W_n})^2\right)+n(n-1)\E(W_1-\bar{W_n})(W_2-\bar{W_n}).
\end{eqnarray*}
Thus,
$$v_W^2=\E\left((W_1-\bar{W_n})^2\right)=-(n-1)\E(W_1-\bar{W_n})(W_2-\bar{W_n}).$$
Hence,
\begin{eqnarray}
p_W(\Lambda)&=&\sum_{\lambda\in \Lambda}\frac{\E_W\left([(P^W_n-\bar{W_n}P_n)(\psi_{\lambda})]^2\right)}{v_W^2}=\sum_{\lambda\in  \Lambda}\frac{\E_W\left([(P^W_n-\bar{W_n}P_n)(\psi_{\lambda}-P_s\psi_{\lambda})]^2\right)}{v_W^2}\nonumber\\
&=&\sum_{\lambda\in  \Lambda}\E_W\left(\frac1{n^2}\sum_{i,j=1}^n\frac{(W_i-\bar{W_n})(W_j-\bar{W_n})}{v_W^2}(\psi_{\lambda}(X_i)-P_s\psi_{\lambda})(\psi_{\lambda}(X_j)-P_s\psi_{\lambda})\right)\nonumber
\end{eqnarray}
\begin{eqnarray}
p_W(\Lambda)&=&\frac1{n^2}\sum_{\lambda\in\Lambda}\sum_{i=1}^n\frac{\E\left((W_i-\bar{W_n})^2\right)}{v_W^2}(\psi_{\lambda}(X_i)-P_s\psi_{\lambda})^2\nonumber\\
&&+\frac1{n^2}\sum_{\lambda\in\Lambda}\sum_{i\neq j=1}^n\frac{\E(W_i-\bar{W_n})(W_j-\bar{W_n})}{v_W^2}(\psi_{\lambda}(X_i)-P_s\psi_{\lambda})(\psi_{\lambda}(X_j)-P_s\psi_{\lambda})\nonumber\\
&=&\frac{1}{n}\left(P_nT(\Lambda)-U_s(\Lambda)\right).\label{RW}
\end{eqnarray}
On the other hand, easy algebra leads to 
\begin{equation*}
\left\|s_m-\hat{s}_m\right\|_2^2=\sum_{\lambda\in \Lambda}\left([(P_n-P_s)(\psi_{\lambda})]^2\right)=\frac 1n\left(P_nT(\Lambda)+(n-1)U_s(\Lambda)\right).
\end{equation*}
Thus, we have $\left\|s_m-\hat{s}_m\right\|_2^2-p_W(\Lambda)=U_s(\Lambda).$ 
\subsection{Proof of Lemma \ref{lem:concU}}
We apply Theorem 3.4 in Houdr\'e $\&$ Reynaud-Bouret \cite{HRB03}. For all $x>0$
\begin{equation}\label{Ustat1}
\p_s\left(\xi U(\Lambda)>\frac1{n^2}\left(5.7B_1\sqrt{x}+8B_2x+384B_3x^{3/2}+1020B_4x^2\right)\right)\leq ee^{-x},
\end{equation}
where 
\begin{eqnarray*}
&U(x,y)=\sum_{\lambda\in \Lambda}(\psi_{\lambda}(x)-P_s\psi_{\lambda})(\psi_{\lambda}(y)-P_s\psi_{\lambda}),&\\
&B_1^2= n^2\E\left[\left(U(X_1,X_2)\right)^2\right],\;B_3^2=n\sup_x\E\left[\left(U(x,X_2)\right)^2\right],\;B_4=\sup_{x,y}U(x,y),&
\end{eqnarray*}
\begin{equation*}
B_2=\sup\left\{\left|\E\sum_{i=1}^n\sum_{j=1}^{i-1}U(X_1,X_2)\alpha_i(X_1)\beta_j(X_2)\right|,\;\E\sum_{i=1}^n\alpha_i^2(X_1)\leq 1,\;\E\sum_{j=1}^n\beta_j^2(X_1)\leq 1\right\}.
\end{equation*}
From Cauchy-Schwarz inequality, for all real numbers $(b_{\lambda})_{\lambda\in\Lambda}$
\begin{equation}\label{QCSI}
\sum_{\lambda\in\Lambda}b_{\lambda}^2=\left(\sup_{\sum a_{\lambda}^2\leq 1}\sum_{\lambda\in\Lambda}a_{\lambda}b_{\lambda}\right)^2.
\end{equation}
In particular, since the system $(\psi_{\lambda})_{\lambda\in\Lambda}$ is orthonormal, for all $x$ in $\R$, $T(\Lambda)=(\sup_{t\in B(\Lambda)}(t-P_st))^2$. Thus 
\begin{equation}\label{ptm}
\left\|T(\Lambda)\right\|_{\infty}\leq 2b_{\lambda}^2.
\end{equation}
Let us now evaluate $B_1,\;B_2,\;B_3$ and $B_4$.

\noindent
{\bf{Evaluation of $B_1$:}}
\begin{eqnarray}
\frac{B_1^2}{n^2}&=&\sum_{\lambda,\lambda'\in \Lambda}\left(P_s\left((\psi_{\lambda}-P_s\psi_{\lambda})(\psi_{\lambda'}-P_s\psi_{\lambda'})\right)\right)^2\nonumber\\
&=&\sum_{\lambda\in \Lambda}\left(\sup_{\sum a_{\lambda'}^2\leq 1}P_s\left((\psi_{\lambda}-P_s\psi_{\lambda})\left[\sum_{\lambda'\in \Lambda}a_{\lambda'}\psi_{\lambda'}-P_s\left(\sum_{\lambda'\in \Lambda}a_{\lambda'}\psi_{\lambda'}\right)\right]\right)\right)^2\nonumber\\
&=&\sum_{\lambda\in \Lambda}\left(\sup_{t\in B(\Lambda)}P_s\left((\psi_{\lambda}-P_s\psi_{\lambda})(t-P_st)\right)\right)^2\leq D_{s,\Lambda}v_{s,\Lambda}^2,\nonumber
\end{eqnarray}
where we use successively the independence of $X_1$ and $X_2$, Inequality (\ref{QCSI}), the orthonormality of the system $(\psi_{\lambda})_{\lambda\in\Lambda}$ and Cauchy-Schwarz inequality.
Thus we obtain
\begin{equation}\label{b1}
B_1\leq nv_{s,\Lambda}\sqrt{D_{s,\Lambda}}.
\end{equation}
{\bf{Evaluation of $B_2$:}}
For all real numbers $y,z$, we have $2yz\leq y^2+z^2$, thus, for all $i,j$ in $\{1,...,n\}$,
\begin{align*}
2P_s\left((\psi_{\lambda}-P_s\psi_{\lambda})\alpha_i\right)&P_s\left((\psi_{\lambda'}-P_s\psi_{\lambda'})\beta_j\right)\\
&\leq\left(P_s\left((\psi_{\lambda}-P_s\psi_{\lambda})\alpha_i\right)\right)^2+\left(P_s\left((\psi_{\lambda'}-P_s\psi_{\lambda'})\beta_j\right)\right)^2.
\end{align*}
We apply (\ref{QCSI}) with $b_{\lambda}=P_s\left((\psi_{\lambda}-P_s\psi_{\lambda})\alpha_i\right)$, since the system $(\psi_{\lambda})_{\lambda\in\Lambda}$ is orthonormal, for all $i$ in $\{1,...,n\}$,
\begin{equation*}
\sum_{\lambda\in\Lambda}\left(P_s\left((\psi_{\lambda}-P_s\psi_{\lambda})\alpha_i\right)\right)^2=\left(\sup_{t\in B(\Lambda)}P_s(t-P_st)\alpha_i\right)^2\leq v_{s,\Lambda}^2P_s\alpha_i^2.
\end{equation*}
Since $\sum_{i=1}^nP_s\alpha_i^2\leq 1$ we deduce that
\begin{equation*}
\sum_{i,j=1}^n\sum_{\lambda\in\Lambda}\left(P_s\left((\psi_{\lambda}-P_s\psi_{\lambda})\alpha_i\right)\right)^2\leq nv_{s,\Lambda}^2.
\end{equation*}
The same inequality holds for $\beta_j$, thus we obtain
\begin{equation}\label{b2}
B_2\leq nv_{s,\Lambda}^2.
\end{equation}
{\bf{Evaluation of $B_3$:}}
For all $x$ in $\R$, $\E[(U(x,X_2))^2]$ is the variance of the function $t_x=\sum_{\lambda\in \Lambda}(\psi_{\lambda}(x)-P_s\psi_{\lambda})\psi_{\lambda}$. $t_x$ is a function in the linear space $S$ spanned by the $(\psi_{\lambda})_{\lambda\in \Lambda}$ and, from inequality (\ref{QCSI}),
$$\left\|t_x\right\|_2^2=\sum_{\lambda\in \Lambda}(\psi_{\lambda}(x)-P_s\psi_{\lambda})^2=\left(\sup_{t\in B(\Lambda)}(t(x)-P_st)\right)^2\leq 2b_\Lambda^2.$$
Thus $\E[(U(x,X_2))^2]=\Var(t_x(X))= 2b_\Lambda^2\Var(t_x(X)/b_\Lambda)\leq 2b_\Lambda^2v_{s,\Lambda}^2.$ Thus
\begin{equation}\label{b3}
B_3\leq \sqrt{2n}b_\Lambda v_{s,\Lambda}.
\end{equation}
{\bf{Evaluation of $B_4$:}}
We apply Cauchy-Schwarz inequality and we obtain
\begin{equation}\label{b4}
B_4\leq \left\|T(\Lambda)\right\|_{\infty}\leq 2b_\Lambda^2.
\end{equation}
Let $\Omega_x^c$ be the event defined by inequality (\ref{Ustat1}). From (\ref{b1}), (\ref{b2}), (\ref{b3}) and (\ref{b4}). On $\Omega_x$, 
\begin{equation*}
\xi U_s(\Lambda)\leq\frac{5.7v_{s,\Lambda}\sqrt{D_{s,\Lambda}x}}{n}+\frac{8v_{s,\Lambda}^2x}n+384\sqrt{2}v_{s,\Lambda}b_\Lambda\left(\frac{x}{n}\right)^{3/2}+2040b_\Lambda\left(\frac{x}{n}\right)^2.
\end{equation*}

\subsection{Proof of Lemma \ref{lem:vbD}}
It comes from Assumption {\bf H2} that 
$$b_{\Lambda}\leq C_{1}\sqrt{d}.$$
It comes from (\ref{QCSI}) that
$$D_{s,\Lambda}\leq \sum_{\lambda\in\Lambda}P_s(\psi_\lambda^2)=P_s\left[\left(\sup_{t\in B(\Lambda)}t\right)^2\right]\leq\left\|\sup_{t\in B(\Lambda)}t\right\|_\infty^2\leq C_1^2d.$$
$v_{s,\Lambda}^2\leq \sup_{t\in B(\Lambda)}P_st^2$, thus
$$v_{s,\Lambda}^2\leq b_\Lambda^2\leq C_1^2d,\;v_{s,\Lambda}^2\leq \left\|s\right\|_\infty \sup_{t\in B(\Lambda)}\|t\|^2=\left\|s\right\|_\infty.$$
Finally, for all $t$ in $B(\Lambda)$,
$$P_st^2\leq \left\|t\right\|_\infty P_s|t|\leq\left\|t\right\|_\infty\|t\|\left\|s\right\|\leq C_{1}\sqrt{d}\left\|s\right\|.$$
Thus $v_{s,\Lambda}^2\leq C_{1}\sqrt{d}\left\|s\right\|.$

\bibliographystyle{plain}
\bibliography{bibliolerasle}

\end{document}